\documentclass[article]{interact}

\usepackage{graphicx}
\usepackage{amsmath}
\usepackage{amssymb}
\usepackage{fancyhdr}
\usepackage{color}
\usepackage{algorithm,algorithmic}
\usepackage{subfigure}
\usepackage{multirow}
\usepackage{rotating}
\usepackage{tabularx}
\usepackage{booktabs}
\usepackage{lscape}
\usepackage{longtable}
\usepackage{url}

\usepackage[numbers,sort&compress]{natbib}
\bibpunct[, ]{[}{]}{,}{n}{,}{,}

\theoremstyle{plain}
\newtheorem{theorem}{Theorem}[section]
\newtheorem{lemma}[theorem]{Lemma}

\theoremstyle{definition}

\theoremstyle{remark}
\newtheorem{remark}{Remark}

\begin{document}

\articletype{Article}

\title{Gradient methods exploiting spectral properties}

\author{
\name{Yakui Huang\textsuperscript{a}, Yu-Hong Dai\textsuperscript{b*},\thanks{*Corresponding author. Email: dyh@lsec.cc.ac.cn} Xin-Wei Liu,\textsuperscript{a}
and Hongchao Zhang\textsuperscript{c}
}
\affil{\textsuperscript{a}Institute of Mathematics, Hebei University of Technology, Tianjin, China;
\textsuperscript{b}LSEC, Academy of Mathematics and Systems Science, Chinese Academy of Sciences, Beijing, China;
\textsuperscript{c}Department of Mathematics, Louisiana State University, Baton Rouge, LA 70803-4918, USA}
}

\maketitle

\begin{abstract}
We propose a new stepsize for the gradient method. It is shown that this new stepsize will converge to the reciprocal of the largest eigenvalue of
the Hessian, when Dai-Yang's asymptotic optimal gradient method (Computational Optimization and Applications, 2006, 33(1): 73-88) is applied for minimizing quadratic objective functions.
Based on this spectral property, we develop a monotone gradient method that takes a certain number of steps using the asymptotically  optimal stepsize by Dai and Yang,
and then follows by some short steps associated with this new stepsize. By employing one step retard of the asymptotic optimal stepsize, a nonmonotone variant of this method is also proposed.
Under mild conditions, $R$-linear convergence of the proposed methods is established for minimizing quadratic functions.
In addition, by combining gradient projection techniques and adaptive
nonmonotone line search, we further extend those methods for general bound constrained optimization.
Two variants of gradient projection methods combining with the Barzilai-Borwein stepsizes are also proposed.
Our numerical experiments on both quadratic and bound constrained optimization  indicate that the new proposed strategies and methods are very effective.
\end{abstract}


\begin{amscode}
90C20; 90C25; 90C30
\end{amscode}
\section{Introduction}

We consider the problem of minimizing a convex quadratic function
\begin{equation}\label{eqpro}
  \min~~f(x)=\frac{1}{2}x^TAx-b^Tx
\end{equation}
and its extensions on bound constrained optimization,
where $b\in \mathbb{R}^n$ and $A\in \mathbb{R}^{n\times n}$ is symmetric positive definite with eigenvalues $0<\lambda_1\leq\lambda_2\leq\ldots\leq\lambda_n$ and condition number $\kappa=\frac{\lambda_n}{\lambda_1}$.
This problem \eqref{eqpro} is one of the simplest non-trivial non-linear programming problems and efficiently solving \eqref{eqpro}  is usually a pre-requisite for a method to be generalized for solving more general optimization. In addition, various optimization problems arising in many applications including machine learning \cite{cortes1995support}, sparse reconstruction \cite{figueiredo2007gradient}, nonnegative matrix factorization \cite{huang2015quadratic,lee1999learning} can be formulated as the form of \eqref{eqpro}, possibly with the addition of regularization or bound constraints.

The simplest and easily implemented method for solving \eqref{eqpro} is the gradient method, which updates the iterates by
\begin{equation}\label{eqitr}
  x_{k+1}=x_k-\alpha_kg_k,
\end{equation}
where $g_k=\nabla f(x_k)$ and $\alpha_k>0$ is the stepsize determined by different strategies.

The classic steepest descent (SD) method for solving \eqref{eqpro}  can be dated back to Cauchy \cite{cauchy1847methode}, who suggested to compute the stepsize by exact line search:
\begin{equation}\label{sd}
  \alpha_k^{SD}=\arg\min_{\alpha \in \mathbb{R}}~f(x_k-\alpha g_k)=\frac{g_{k}^Tg_{k}}{g_{k}^TAg_{k}}.
\end{equation}
It has been shown that the method converges linearly \cite{akaike1959successive} with $Q$-linear rate $\frac{\kappa-1}{\kappa+1}$.
Thus, the SD method can be very slow especially when the condition number is large. Further analysis shows that the gradients will asymptotically perform zigzag between two orthogonal directions
in the subspace spanned by the two eigenvectors corresponding to $\lambda_1$ and $\lambda_n$, see \cite{forsythe1968asymptotic,nocedal2002behavior} for more details.



In 1988, from the view of quasi-Newton method, Barzilai and Borwein \cite{Barzilai1988two} designed a method, called BB method, using
the following two ingenious stepsizes that significantly improve the performance of gradient methods:
\begin{equation}\label{sbb}
  \alpha_k^{BB1}=\frac{s_{k-1}^Ts_{k-1}}{s_{k-1}^Ty_{k-1}}, \quad \mbox{and} \quad \alpha_k^{BB2}=\frac{s_{k-1}^Ty_{k-1}}{y_{k-1}^Ty_{k-1}},
\end{equation}
where $s_{k-1}=x_k-x_{k-1}$ and $y_{k-1}=g_k-g_{k-1}$. The BB method was shown to be globally convergent for minimizing general $n$-dimensional strictly convex quadratics \cite{raydan1993Barzilai}
with $R$-linear  convergence rate \cite{dai2002r}. Recently, the BB method and its variants have been successfully extended to general unconstrained problems \cite{raydan1997barzilai}, to constrained optimization problems \cite{birgin2000nonmonotone,huang2016smoothing} and to various applications \cite{huang2015quadratic,jiang2013feasible,liu2011coordinated}. One may see  \cite{birgin2014spectral,dhl2018,di2018steplength,fletcher2005barzilai,yuan2008step} and the references therein.

Let $\{\xi_1,\xi_2,\ldots,\xi_n\}$ be the orthonormal eigenvectors associated with the eigenvalues. Denote the components of $g_k$ along the eigenvectors $\xi_i$ by $\mu_i^k$, $i=1,\ldots,n$, i.e.,
\begin{equation*}\label{gkmu}
  g_k=\sum_{i=1}^n\mu_i^k\xi_i.
\end{equation*}
The above decomposition of gradient $g_k$ together with the update rule \eqref{eqitr} give
\begin{equation*}
  g_{k+1}=g_k-\alpha_kAg_k=\prod_{j=1}^k(I-\alpha_jA)g_1=\sum_{i=1}^n\mu_i^{k+1}\xi_i,
\end{equation*}
where
\begin{equation*}\label{muitr}
  \mu_i^{k+1}=\mu_i^k(1-\alpha_k\lambda_i)=\mu_i^1\prod_{j=1}^k(1-\alpha_j\lambda_i).
\end{equation*}
This relation implies that the closer $\alpha_k$ to $\frac{1}{\lambda_i}$, the smaller $|\mu_i^{k+1}|$ would be. In addition, if $\mu_i^{k}=0$, the corresponding component will vanish at all subsequent iterations.

Since the SD method will asymptotically zigzag between $\xi_1$ and $\xi_n$, a natural way to break the zigzagging pattern is to eliminate the component $\mu_1^k$ or $\mu_n^k$, which can be achieved by employing a stepsize approximating $\frac{1}{\lambda_1}$ or $\frac{1}{\lambda_n}$. One seminal work in this line of research is due to
Yuan \cite{dai2005analysis,yuan2006new}, who derived the following stepsize by imposing finite termination for minimizing two-dimensional convex quadratics:
\begin{equation}\label{syv}
  \alpha_k^{Y}=\frac{2}{\sqrt{(1/\alpha_{k-1}^{SD}-1/\alpha_k^{SD})^2+4\|g_k\|^2/(\alpha_{k-1}^{SD}\|g_{k-1}\|)^2}+(1/\alpha_{k-1}^{SD}+1/\alpha_k^{SD})}.
\end{equation}
Based on \eqref{syv}, Dai and Yuan \cite{dai2005analysis} further suggested a new gradient method whose stepsize is given by
\begin{equation}\label{sdy}
  \alpha_k^{DY}=\left\{
                  \begin{array}{ll}
                    \alpha_k^{SD}, & \hbox{if mod($k$,4)$<2$;} \\
                    \alpha_k^{Y}, & \hbox{otherwise.}
                  \end{array}
                \right.
\end{equation}
The DY method \eqref{sdy} keeps monotonicity and appears better than the nonmonotone BB method \cite{dai2005analysis}. It is shown by De Asmundis et al. \cite{de2014efficient} that the stepsize $\alpha_k^{Y}$ converges to $\frac{1}{\lambda_n}$ if the SD method is applied to solve problem \eqref{eqpro}. That is, occasionally employing the stepsize $\alpha_k^{Y}$ along the SD method will enhance the elimination of the component $\mu_n$. Recently, Gonzaga and Schneider \cite{gonzaga2016steepest} suggest a monotone method with all stepsizes of the form \eqref{sd}.
Their method approximates $\frac{1}{\lambda_n}$ by a short stepsize calculated by replacing $g_k$ in \eqref{sd} with $\tilde{g}=(I-\eta A)g_k$ for some large scalar $\eta$.

Nonmonotone gradient methods exploiting spectral properties have been developed as well. Frassoldati et al. \cite{frassoldati2008new} developed a new short stepsize by maximizing the next SD stepsize $\alpha_{k+1}^{SD}$. They further suggested a method, called ABB$_{\min2}$, which tries to enforce BB1 stepsizes close to $\frac{1}{\lambda_1}$ by using short stepsizes to eliminate gradient components associated with large eigenvalues. More recently, based on the favorable property of $\alpha_k^{Y}$, De Asmundis et al. \cite{de2014efficient} suggested to reuse it in a cyclic fashion after a certain number of SD steps. Precisely, their approach, referred to as the SDC method, employs the stepsize
\begin{equation}\label{ssdc}
  \alpha_k^{SDC}=\left\{
                  \begin{array}{ll}
                    \alpha_k^{SD}, & \hbox{if mod$(k,h+s)<h$;} \\
                    \alpha_t^{Y}, & \hbox{otherwise, with $t=\max\{i\leq k: \textrm{mod}(i,h+s)=h$\},}
                  \end{array}
                \right.
\end{equation}
where $h\geq2$ and $s\geq1$. They also proposed a monotone version of \eqref{ssdc} by imposing safeguards on the stepsizes.

One common character of the aforementioned methods is making use of spectral properties of the stepsizes. The recent study \cite{di2018steplength} points out that gradient methods using long and short stepsizes that attempt to exploit the spectral properties have generally better numerical performance for minimizing both quadratic and general nonlinear objective functions.
For more works on gradient methods, see \cite{dai2003alternate,dai2005projected,de2014efficient,di2018steplength,frassoldati2008new,gonzaga2016steepest,yuan2008step,zhou2006gradient}.
In \cite{dai2006new2}, Dai and Yang  introduced a gradient method with a new stepsize that possesses similar spectral property as $\alpha_k^{Y}$. More specifically, their stepsize is given by
\begin{equation}\label{saopt}
  \alpha_k^{AOPT}=\frac{\|g_k\|}{\|Ag_k\|},
\end{equation}
which asymptotically converges to $\frac{2}{\lambda_1+\lambda_n}$, that is in some sense an optimal stepsize since it minimizes $\|I-\alpha A\|$ over the stepsize $\alpha$ \cite{dai2006new2,elman1994inexact}. Since $\alpha_k^{AOPT}\leq\alpha_k^{SD}$, the Dai-Yang method \eqref{saopt} is monotone but without using exact line searches. In addition, the method converges $Q$-linearly with the same rate as the SD method. More importantly, by applying this method it is possible to recover the eigenvectors $\xi_1$ and $\xi_n$.

In this paper, based on the Dai-Yang method \eqref{saopt}, we propose a new stepsize to exploit the spectral property. Particularly, our new stepsize is given by
\begin{equation}\label{baralp}
  \bar{\alpha}_{k}=\frac{d_k^Td_k}{d_k^TAd_k},
\end{equation}
where
\begin{equation}\label{dk}
  d_k=\frac{g_{k-1}}{\|g_{k-1}\|}-\frac{g_{k}}{\|g_{k}\|}.
\end{equation}
We show that the stepsize $\bar{\alpha}_k$ asymptotically converges to $\frac{1}{\lambda_n}$ if the Dai-Yang method \eqref{saopt} is applied to problem \eqref{eqpro}. Therefore, the stepsize $\bar{\alpha}_k$ is helpful in eliminating the gradient component $\mu_n$. Thanks to this desired property, we are able to develop a new efficient gradient method by taking a certain number of steps using
 the asymptotically optimal stepsize $\alpha_k^{AOPT}$ followed by some short steps, which are determined by the smaller one of $\alpha_k^{AOPT}$ and $\bar{\alpha}_{k-1}$. Thus, this method
is a monotone method without using exact line searches.
We also construct a nonmonotone variant of the method which simply use the stepsize $\alpha_k^{AOPT}$ with one step retard. $R$-linear convergence of the proposed methods is established for minimizing
strongly convex quadratic functions.  In addition, by combining gradient projection techniques  and the adaptive nonmonotone line search in \cite{dai2001adaptive}, we further to
extend those proposed methods for general bound constrained optimization. Two variants of gradient projection methods combining with the BB stepsizes are also proposed.
Our numerical comparisons with DY (\ref{sdy}), ABB$_{\min 2}$ \cite{frassoldati2008new} and SDC (\ref{ssdc}) methods on minimizing quadratic functions indicate the proposed strategies and methods
are very effective. Moreover, our numerical comparisons with the spectral projected gradient (SPG) method \cite{birgin2000nonmonotone,birgin2014spectral} on solving bound constrained optimization
problems from the CUTEst collection \cite{gould2015cutest} also highly suggest the potential benefits of extending the strategies and methods in the paper for more general
 large-scale bound constrained optimization.

The paper is organized as follows. In Section \ref{sec2}, we analyze the asymptotic spectral property of the new stepsize $\bar{\alpha}_k$ and propose our new methods based on this spectral property.
 In Section \ref{sec3}, we show that the new proposed methods have $R$-linear convergence for minimizing strongly convex quadratic functions.
 We generalize the proposed ideas and methods for bound constrained optimization in Section \ref{sec4}.
Some numerical comparisons of our new methods on solving both quadratic and bound constrained optimization problems are shown in Section \ref{sec5}.
Finally, in Section \ref{sec6} we give some concluding remarks.

\section{Method for quadratics}\label{sec2}
In this section we first analyze the spectral property of the stepsize $\bar{\alpha}_k$ and then propose our new gradient methods.

\subsection{Spectral property of $\bar{\alpha}_k$}

We first recall some important properties of the Dai-Yang method \eqref{saopt}.
\begin{lemma}\cite{dai2006new2}\label{lmdy}
For any starting point $x_1$ satisfying
\begin{equation*}
  \mu_1^1\neq0,~~\mu_n^1\neq0,
\end{equation*}
let $\{x_k\}$ be the iterations generated by the method \eqref{saopt}. Then we have that
\begin{equation*}
  \lim_{k\rightarrow\infty}\alpha_{k}^{AOPT}=\frac{2}{\lambda_1+\lambda_n}.
\end{equation*}
Furthermore,
\begin{equation*}
  \lim_{k\rightarrow\infty}\frac{\mu_i^{2k-1}}{\sqrt{\sum_{j=1}^n(\mu_j^{2k-1})^2}}=\left\{
  \begin{array}{ll}
   {\rm{sign}}(\mu_1^1)\sqrt{c_1}, &\hbox{if $i=1$;} \\
    0,&\hbox{if $i=2,\ldots,n-1$;} \\
    {\rm{sign}}(\mu_n^1)\sqrt{c_2},&\hbox{if $i=n$,}
  \end{array}
\right.
\end{equation*}
and
\begin{equation*}
  \lim_{k\rightarrow\infty}\frac{\mu_i^{2k}}{\sqrt{\sum_{j=1}^n(\mu_j^{2k})^2}}=\left\{
  \begin{array}{ll}
   {\rm{sign}}(\mu_1^1)\sqrt{c_1}, &\hbox{if $i=1$;} \\
    0,&\hbox{if $i=2,\ldots,n-1$;} \\
    -{\rm{sign}}(\mu_n^1)\sqrt{c_2},&\hbox{if $i=n$,}
  \end{array}
\right.
\end{equation*}
which indicates that
\begin{equation*}
  \lim_{k\rightarrow\infty}\frac{g_{k-1}}{\|g_{k-1}\|}+\frac{g_{k}}{\|g_{k}\|}=2\rm{sign}(\mu_1^1)\sqrt{c_1}\xi_1
\end{equation*}
and
\begin{equation*}
  \lim_{k\rightarrow\infty}\frac{g_{k-1}}{\|g_{k-1}\|}-\frac{g_{k}}{\|g_{k}\|}=\pm2\sqrt{c_2}\xi_n,
\end{equation*}
where
\begin{equation*}
  c_1=\frac{\lambda_1+3\lambda_n}{4(\lambda_1+\lambda_n)},~~
  c_2=\frac{3\lambda_1+\lambda_n}{4(\lambda_1+\lambda_n)}.
\end{equation*}
\end{lemma}

Lemma \ref{lmdy} indicates that the method \eqref{saopt} asymptotically conducts its searches in the two-dimensional subspace spanned by $\xi_1$ and $\xi_n$.
So, in order to accelerate the convergence, we could employ some stepsizes approximating $\frac{1}{\lambda_1}$ or $\frac{1}{\lambda_n}$ to eliminate the component $\mu_1^k$ or $\mu_n^k$. Note that the vector $d_k$ given in (\ref{dk}) tends to align in the direction of the eigenvector $\xi_n$. Hence, if we take some consecutive gradient steps with stepsize
$\alpha_{k}^{AOPT}$ so that $d_k\approx\pm2\sqrt{c_2}\xi_n$, the stepsize $\bar{\alpha}_k$ will be an approximation of $\frac{1}{\lambda_n}$. The next theorem provides theoretical
justification for this strategy.

\begin{theorem}\label{th1}
  Under the conditions in Lemma \ref{lmdy}, let $\{g_k\}$ be the sequence generated by applying the method \eqref{saopt} to problem \eqref{eqpro}. Then we have
\begin{equation*}\label{spes2}
  \lim_{k\rightarrow\infty}\bar{\alpha}_k=\frac{1}{\lambda_n}.
\end{equation*}
\end{theorem}
\begin{proof}
From the definition \eqref{dk} of $d_k$, we have
\begin{equation}\label{ddk}
  d_k^Td_k=2-2\frac{g_{k-1}^Tg_k}{\|g_{k-1}\|\|g_k\|}
\end{equation}
and
\begin{align}\label{dadk}
  d_k^TAd_k&=\frac{g_{k-1}^TAg_{k-1}}{\|g_{k-1}\|^2}+
  \frac{g_k^TAg_k}{\|g_k\|^2}
  -2\frac{g_{k-1}^TAg_k}{\|g_{k-1}\|\|g_k\|},
\end{align}
which indicate that
\begin{align}\label{limddk}
  \lim_{k\rightarrow\infty}d_k^Td_k&=
  2-2\lim_{k\rightarrow\infty}\sum_{j=1}^n\frac{\mu_i^{2k-1}}{\sqrt{\sum_{j=1}^n(\mu_j^{2k-1})^2}}\frac{\mu_i^{2k}}{\sqrt{\sum_{j=1}^n(\mu_j^{2k})^2}}
  \nonumber\\
  &=2-2(c_1-c_2)
\end{align}
and
\begin{align}\label{limdadk}
  \lim_{k\rightarrow\infty}d_k^TAd_k&=
  \lim_{k\rightarrow\infty}\frac{\sum_{j=1}^n\lambda_i(\mu_i^{2k-1})^2}{\sum_{j=1}^n(\mu_j^{2k-1})^2}
  +\frac{\sum_{j=1}^n\lambda_i(\mu_i^{2k})^2}{\sum_{j=1}^n(\mu_j^{2k})^2}\nonumber\\
  &-2\sum_{j=1}^n\lambda_i\frac{\mu_i^{2k-1}}{\sqrt{\sum_{j=1}^n(\mu_j^{2k-1})^2}}\frac{\mu_i^{2k}}{\sqrt{\sum_{j=1}^n(\mu_j^{2k})^2}}
  \nonumber\\
  &=2(\lambda_1c_1+\lambda_nc_2)-2(\lambda_1c_1-\lambda_nc_2)\nonumber\\
  &=4\lambda_nc_2.
\end{align}
It follows from \eqref{limddk}, \eqref{limdadk} and the definition \eqref{baralp} of $\bar{\alpha}_{k}$ that
\begin{align*}\label{balpk}
  \lim_{k\rightarrow\infty}\bar{\alpha}_{k}
  &=\lim_{k\rightarrow\infty}\frac{d_k^Td_k}{d_k^TAd_k}
  =\frac{2-2(c_1-c_2)}{4\lambda_nc_2}\\
&=\frac{4(\lambda_1+\lambda_n)-(2\lambda_n-2\lambda_1)}{2\lambda_n(3\lambda_1+\lambda_n)}
=\frac{1}{\lambda_n}.
\end{align*}
This completes the proof.
\end{proof}

Using the same argument as those in Theorem \ref{th1}, we can also get the following result.
\begin{theorem}\label{th2}
  Under the conditions in Lemma \ref{lmdy}, let $\{g_k\}$ be the sequence generated by applying the method \eqref{saopt} to problem \eqref{eqpro}, we have
\begin{equation*}
  \lim_{k\rightarrow\infty}\hat{\alpha}_k=\frac{1}{\lambda_1},
\end{equation*}
where
\begin{equation*}\label{hatalp}
  \hat{\alpha}_{k}=\frac{\hat{d}_k^T\hat{d}_k}{\hat{d}_k^TA\hat{d}_k}
\end{equation*}
with
\begin{equation*}\label{hatdk}
  \hat{d}_k=\frac{g_{k-1}}{\|g_{k-1}\|}+\frac{g_{k}}{\|g_{k}\|}.
\end{equation*}
\end{theorem}

\subsection{The algorithm}

Theorems \ref{th1} and \ref{th2} in the former subsection provide us the possibility of employing the two stepsizes $\hat{\alpha}_{k}$ and $\bar{\alpha}_{k}$
to significantly reduce the gradient components $\mu_1^k$ and $\mu_n^k$. However, the following example shows some negative aspects of using $\hat{\alpha}_{k}$.
Particularly, we applied the method \eqref{saopt} to a problem of \eqref{eqpro} with
\begin{equation}\label{tp1}
  A=diag\{a_1,a_2,\ldots,a_n\} \quad \mbox{and} \quad b=0,
\end{equation}
where $a_1=1$, $a_n=n$ and $a_i$ is randomly generated in $(1,n)$, $i=2,\ldots,n-1$. Figure \ref{appstep1} presents the result of an instance with $n=1,000$. We can see that $\bar{\alpha}_k$ approximates $\frac{1}{\lambda_n}$ with satisfactory accuracy in a few iterations. However,  $\hat{\alpha}_{k}$ converges to $\frac{1}{\lambda_1}$ very slowly in the first few hundreds of iterations.
although we did observe that after 1,000 iterations the value of $|\hat{\alpha}_{k}-\frac{1}{\lambda_1}|$ is reduced by a factor of 0.01.


\begin{figure}[h]
  \centering
  \includegraphics[width=0.75\textwidth,height=0.48\textwidth]{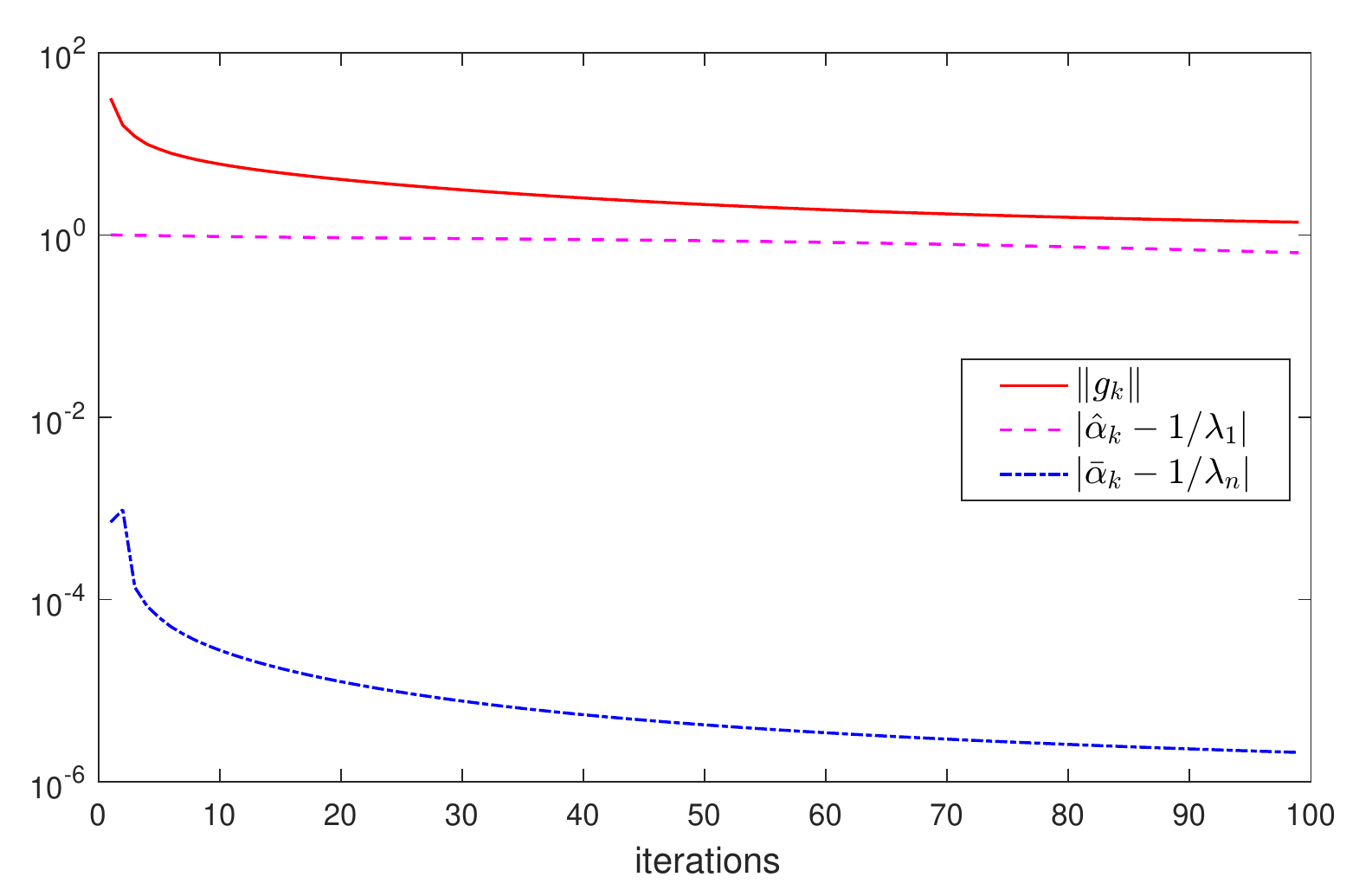}\\
  \caption{Problem \eqref{tp1}: convergence history of the sequence $\{\hat{\alpha}_k\}$ and $\{\bar{\alpha}_k\}$ for the first 100 iterations of the method \eqref{saopt}.}\label{appstep1}
\end{figure}



Now we would like to give a rough explanation of the above phenomenon. Since the gradient norm decreases very slowly, by the update rule \eqref{eqitr}, we have
\begin{equation*}\label{aprgk}
  \frac{g_{k}^{(i)}}{\|g_{k}\|}=(1-\alpha_k\lambda_i)\frac{g_{k-1}^{(i)}}{\|g_{k-1}\|}\frac{\|g_{k-1}\|}{\|g_{k}\|}
  \approx(1-\alpha_k\lambda_i)\frac{g_{k-1}^{(i)}}{\|g_{k-1}\|}.
\end{equation*}
Suppose that $\alpha_k$ satisfies $|1-\alpha_k\lambda_i|\leq1$ for all $i=1,2,\ldots,n$ and $k$ is sufficiently large. This will be true since $\alpha_k^{AOPT}$ approximates $\frac{2}{\lambda_1+\lambda_n}$ as the iteration process goes on. Let $d_k^{(i)}$ be the $i$-th component of $d_k$, i.e.,
\begin{equation*}
  d_k^{(i)}=\frac{g_{k}^{(i)}}{\|g_{k}\|}-\frac{g_{k-1}^{(i)}}{\|g_{k-1}\|}.
\end{equation*}
We consider the following cases:

\textbf{Case 1}. $|1-\alpha_k\lambda_i|\leq0.6$.

By trivial computation we know that after five steps the value of $\frac{|g_{k}^{(i)}|}{\|g_{k}\|}$ is less than 8\% of its initial value and keeps small at all subsequent iterations
Thus, $d_k^{(i)}$ can be neglected.

\textbf{Case 2}. $|1-\alpha_k\lambda_i|>0.6$.

(i) If $1-\alpha_k\lambda_i\geq0.9$, the value of $\frac{g_{k}^{(i)}}{\|g_{k}\|}$ will not change much and thus, $d_k^{(i)}$ may not affect the value of $\bar{\alpha}_k$ too much,
 which indicates that the component can be also neglected.

(ii) If $0.6<1-\alpha_k\lambda_i<0.9$, then $0.1<\alpha_k\lambda_i<0.4$, which implies that $d_k^{(i)}$ will be small in a few iterations and is safe to be abandoned.



(iii) If $1-\alpha_k\lambda_i<-0.6$, the value of $\frac{g_{k}^{(i)}}{\|g_{k}\|}$ changes signs and hence, $|d_k^{(i)}|$ may get significant increase.

The above analysis shows that $\bar{\alpha}_k$ will be mostly determined by the components corresponding to those eigenvalues in (iii) of Case 2. Notice that the required inequality in (iii) implies that $\lambda_i>\frac{4(\lambda_1+\lambda_n)}{5}$. If $A$ has many large eigenvalues satisfying this condition, $\bar{\alpha}_k$ will be an estimation of
the reciprocal of certain average of large eigenvalues. When $A$ has few such large eigenvalues, $\bar{\alpha}_k$ will be mostly determined
by the gradient components corresponding to the these few largest eigenvalues, which would
yield a good estimation of $\frac{1}{\lambda_n}$. So, $\bar{\alpha}_k$ will approximate $\frac{1}{\lambda_n}$ with satisfactory accuracy in small number of iterations.
 This coincides with our observation in Figure \ref{appstep1}.

For the stepsize $\hat{\alpha}_k$, when $1-\alpha_k\lambda_i>0$, the value of $\frac{g_{k}^{(i)}}{\|g_{k}\|}+\frac{g_{k-1}^{(i)}}{\|g_{k-1}\|}$ may increase even when $1-\alpha_k\lambda_i\leq0.1$.
So, most of the components of the gradient corresponding to those eigenvalues less than $\frac{1}{\alpha_k}$ will affect the value of $\hat{\alpha}_k$. Thus, $\hat{\alpha}_{k}$ would not be a good approximation of $\frac{1}{\lambda_1}$ until those components become very small. Moreover, a rough estimation of $\frac{1}{\lambda_1}$ may yield a large step which will increase most of
 the components of the gradient. As a result, it is impractical to use $\hat{\alpha}_{k}$ for eliminating the gradient component $\mu_1^k$.

Based on the above observations, our method would combine the stepsizes $\bar{\alpha}_k$ and $\alpha_{k}^{AOPT}$. In particular, our method takes $h$ steps with $\alpha_{k}^{AOPT}$ to drive $\bar{\alpha}_{k}$ towards a good approximation of $\frac{1}{\lambda_n}$ and then takes $s$ short steps in the hope of eliminating the corresponding component $\mu_n$.
As $\bar{\alpha}_{k}$ is expected to be short, we resort to $\alpha_{k}^{AOPT}$ if $\bar{\alpha}_{k}$ is relatively large. Hence, more precisely, we take
\begin{equation}\label{news0}
\alpha_{k}=
\begin{cases}
\alpha_{k}^{AOPT},& \text{if $\textrm{mod}(k,h+s)<h$}; \\
\min\{\alpha_{k}^{AOPT},\bar{\alpha}_{k}\},& \text{otherwise}.
\end{cases}
\end{equation}


Notice that, for quadratics, the BB1 stepsize $\alpha_{k}^{BB1}$ is just the former SD stepsize \eqref{sd}. As we know, the BB method performs much better than the SD method \cite{fletcher2005barzilai,yuan2008step}. And gradient methods with retard stepsizes often have better performances \cite{friedlander1998gradient}.
Moreover, it can be seen from Figure \ref{appstep1} that $\bar{\alpha}_{k-1}$ is also a good approximation of $\frac{1}{\lambda_n}$ after about 10 to 20 iterations.
Thus, we also consider to use the retard stepsize $\bar{\alpha}_{k-1}$, i.e.,
\begin{equation} \label{news}
\alpha_{k}=
\begin{cases}
\alpha_{k}^{AOPT},& \text{if $\textrm{mod}(k,h+s)<h$}; \\
\min\{\alpha_{k}^{AOPT},\bar{\alpha}_{k-1}\},& \text{otherwise}.
\end{cases}
\end{equation}
Numerical comparisons  between the methods \eqref{news0} and \eqref{news} in Table \ref{tbcmpdy} show the benefits of using the one stepsize delay.
Table \ref{tbcmpdy} lists the averaged iterations of these two methods on solving $10$ instances of problem \eqref{tp1}, where the condition number of $A$ is $\kappa=10^4$ with $a_1=1$, $a_n=\kappa$, and other diagonal elements are randomly generated in $(1,\kappa)$. The iteration was stopped once the gradient norm is less than an $\epsilon$ factor of its initial value.
We can see that the performance of the method \eqref{news} dominates that of \eqref{news0} for most of the instances.
Another advantage of using the retard stepsize  $\bar{\alpha}_{k-1}$ is that it can be easily extended to more general problems. This will be more clear in Section \ref{sec4}.
\begin{table}[ht!b]
\setlength{\tabcolsep}{0.2ex}
\tbl{Number of averaged iterations of the methods \eqref{news0} and \eqref{news}.}
{\begin{scriptsize}
\begin{tabular}{|c|c|c|c|c|c|c|c|c|c|c|c|}
\hline
 \multicolumn{1}{|c|}{\multirow{2}{*}{method}} &\multicolumn{1}{c|}{\multirow{2}{*}{$\epsilon$}}
 &\multicolumn{10}{c|}{$(h,s)$ for the method}\\
\cline{3-12}
 \multicolumn{1}{|c|}{} &\multicolumn{1}{c|}{}   &$(10,20)$ &$(10,30)$ &$(10,50)$ &$(10,80)$  &$(10,100)$  &$(20,20)$ &$(20,30)$ &$(20,50)$ &$(20,80)$  &$(20,100)$
\\
  \hline
%

\multicolumn{1}{|c|}{\multirow{3}{*}{\eqref{news0}}}
&\multicolumn{1}{c|}{\multirow{1}{*}{$10^{-6}$}}	
&298.9 &287.3 &308.7 &321.4 &336.9 &306.5 &333.5 &331.2 &342.5 &371.7\\
&\multicolumn{1}{c|}{\multirow{1}{*}{$10^{-9}$}}
&711.5 &636.3 &650.2 &606.0 &602.9 &785.7 &680.3 &582.5 &545.7 &655.0\\
&\multicolumn{1}{c|}{\multirow{1}{*}{$10^{-12}$}}
&1048.0 &893.1 &1013.7 &910.3 &816.3 &1225.8 &989.8 &904.1 &773.8 &892.4\\
\hline

\multicolumn{1}{|c|}{\multirow{3}{*}{\eqref{news}}}
&\multicolumn{1}{c|}{\multirow{1}{*}{$10^{-6}$}}	
&318.8 &312.7 &311.5 &350.9 &333.7 &341.0 &337.1 &342.0 &362.2 &349.2\\
&\multicolumn{1}{c|}{\multirow{1}{*}{$10^{-9}$}}
&642.6 &581.3 &566.0 &561.0 &525.3 &650.6 &655.0 &636.5 &539.6 &571.3\\
&\multicolumn{1}{c|}{\multirow{1}{*}{$10^{-12}$}}
&957.6 &789.9 &772.0 &771.0 &752.5 &952.6 &881.1 &839.3 &755.3 &766.4\\
\hline
\end{tabular}
\end{scriptsize}
}\label{tbcmpdy}
\end{table}

\begin{remark}
Although our method \eqref{news} looks like the SDC method \eqref{ssdc}, they differ in the following ways: (i) the method \eqref{news} does not use exact line searches which are necessary for the SDC method; (ii) the method \eqref{news} does not reuse any stepsize while the SDC method uses the same Yuan's stepsize $\alpha_k^{Y}$ for $s$ steps; (iii) the method \eqref{news} is monotone while the SDC method is nonmonotone and its monotone version is obtained by using a safeguard with $2\alpha_{k}^{SD}$.
\end{remark}

\begin{figure}[th]
  \centering
  \includegraphics[width=0.75\textwidth,height=0.48\textwidth]{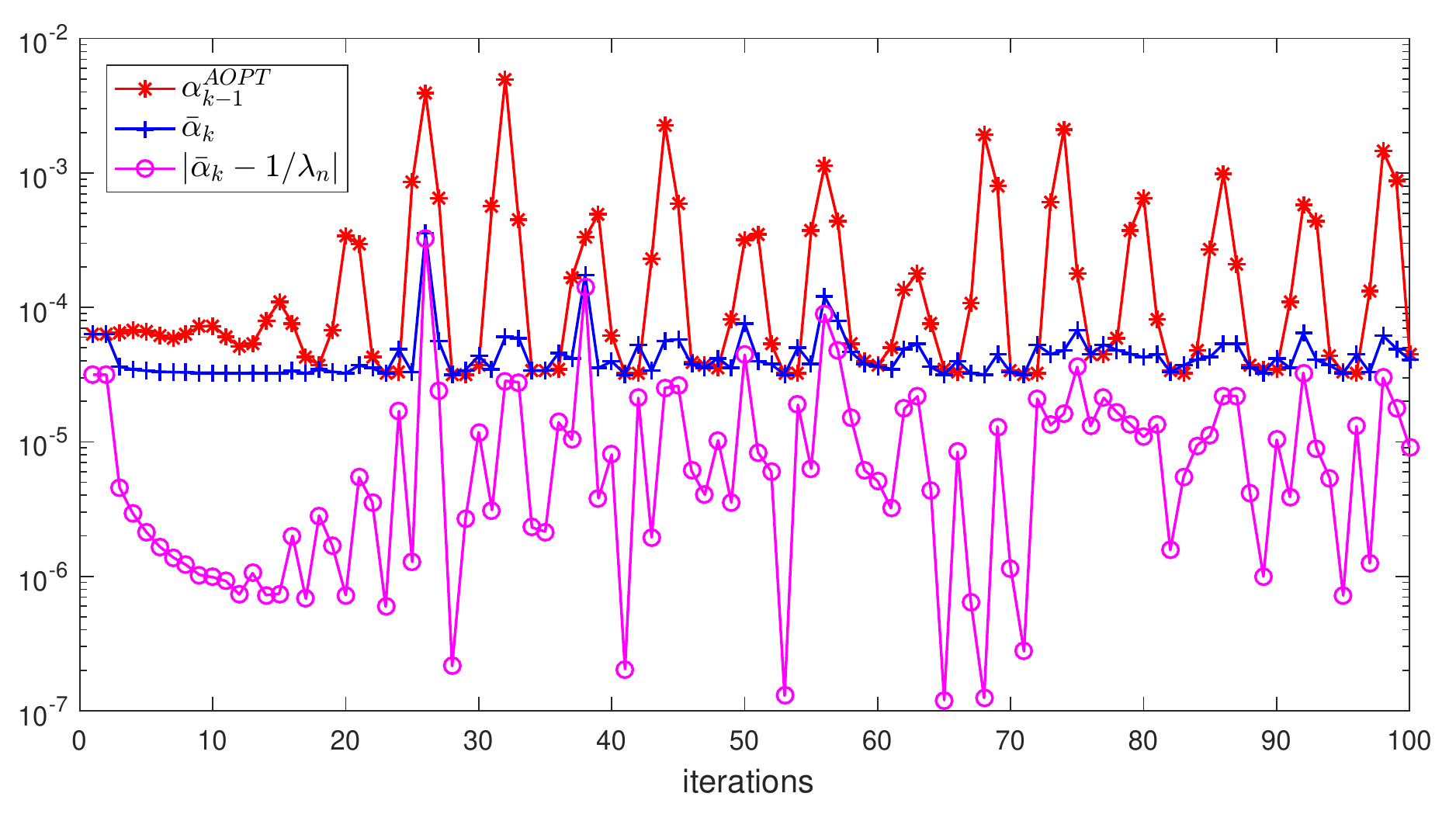}\\
  \caption{Problem \eqref{tp1}: history of the sequence $\{\bar{\alpha}_k\}$ for the first 100 iterations of the gradient method with $\alpha_{k-1}^{AOPT}$.}\label{appstep2}
\end{figure}

%
%

The analysis at the beginning of this subsection indicates that a small $\bar{\alpha}_{k}$ will be generated once its value is mostly determined by the first few largest eigenvalues.
This can be achieved by using short stepsizes such that $|1-\alpha_k\lambda_i|\leq1$ hold for all $i=1,2,\ldots,n$ and several subsequent iterations. In addition, if there exist subsequences $\{\alpha_{k_i}\}$ approximate $\frac{1}{\lambda_i}$ for all but the first few largest eigenvalues, $\bar{\alpha}_{k}$ would be also small.
Notice that the retard stepsize $\alpha_{k-1}^{AOPT}$ is an approximation of some $\frac{1}{\lambda_i}$ and also short in the sense that $\alpha_{k-1}^{AOPT}\leq\alpha_{k-1}^{SD}$.
In fact, we can see from Figure \ref{appstep2} that, when the gradient method with $\alpha_{k-1}^{AOPT}$ is applied to problem \eqref{tp1}, the stepsize $\bar{\alpha}_{k}$
approximates $\frac{1}{\lambda_n}$ with high accuracy if its value is small.
Moreover, some promising numerical results of applying $\alpha_{k-1}^{AOPT}$ are given in \cite{dai2015positive}.
So, motivated by the above observation and analysis, we also suggest the following nonmonotone variant of (\ref{news}):
\begin{equation}\label{news2}
\alpha_{k}=
\begin{cases}
\alpha_{k-1}^{AOPT},& \text{if $\textrm{mod}(k,h+s)<h$}; \\
\min\{\alpha_{k-1}^{AOPT},\bar{\alpha}_{k-1}\},& \text{otherwise}.
\end{cases}
\end{equation}

\section{Convergence}\label{sec3}
In this section, we establish the $R$-linear convergence of the method \eqref{news} and its nonmonotone variant \eqref{news2} for minimizing strongly convex quadratic function.
Since the gradient method \eqref{eqitr} is invariant under translations and rotations when applying to problem \eqref{eqpro}, we make the following assumption throughout the analysis.

\textbf{Assumption 1.} The matrix $A$ is diagonal, i.e.,
\begin{equation}\label{formA}
  A=\textrm{diag}\{\lambda_1,\lambda_2,\cdots,\lambda_n\},
\end{equation}
with $0<\lambda_1<\lambda_2<\cdots<\lambda_n$.

In order to give a unified analysis of the methods \eqref{news} and \eqref{news2}, we recall the following property given by Dai \cite{dai2003alternate}.

\noindent
\textbf{Property (A)} \cite{dai2003alternate}. Suppose that there exist an integer $m$ and positive constants $M_1\geq\lambda_1$ and $M_2$ such that
\begin{itemize}
  \item[(i)] $\lambda_1\leq\alpha_k^{-1}\leq M_1$;
  \item[(ii)] for any integer $l\in[1,n-1]$ and $\epsilon>0$, if $G(k-j,l)\leq\epsilon$ and $(g_{k-j}^{(l+1)})^2\geq M_2\epsilon$ hold for $j\in[0,\min\{k,m\}-1]$, then $\alpha_k^{-1}\geq\frac{2}{3}\lambda_{l+1}$.
\end{itemize}
Here,
\begin{equation*}
  G(k,l)=\sum_{i=1}^l(g_k^{(i)})^2.
\end{equation*}

Dai \cite{dai2003alternate} has proved that if $A$ has the form \eqref{formA} with $1=\lambda_1\leq\lambda_2\leq\cdots\leq\lambda_n$ and the stepsizes of gradient method \eqref{eqitr} have the Property (A), then either $g_k=0$ for some finite $k$ or the sequence $\{\|g_k\|\}$ converges to zero $R$-linearly.
Therefore, in order to establish $R$-linear convergence of the methods \eqref{news} and \eqref{news2}, we only need to show these methods satisfy Property (A).
\begin{theorem}\label{th4}
Suppose that the sequence $\{\|g_k\|\}$ is generated by any of the method \eqref{news} or \eqref{news2} applied to solve problem (\ref{eqpro}) with the matrix $A$ having the form \eqref{formA} and $1=\lambda_1<\lambda_2<\cdots<\lambda_n$. Then either $g_k=0$ for some finite $k$ or the sequence $\{\|g_k\|\}$ converges to zero $R$-linearly.
\end{theorem}
\begin{proof}
We show that the stepsize $\alpha_k$ has Property (A) with $m=2$, $M_1=\lambda_n$ and $M_2=2$.

Clearly, $\lambda_1\leq\alpha_k^{-1}\leq \lambda_n$ for all $k\geq1$. Thus, (i) of Property (A) holds with $M_1=\lambda_n$.

Notice that $\alpha_{k}^{AOPT}\leq\alpha_{k}^{SD}$. For the method \eqref{news2}, we have that
\begin{equation}\label{eqalfnew2}
  \alpha_{k}^{-1}\geq(\alpha_{k-1}^{AOPT})^{-1}\geq(\alpha_{k-1}^{SD})^{-1}.
\end{equation}
Suppose that $G(k-j,l)\leq\epsilon$ and $(g_{k-j}^{(l+1)})^2\geq 2\epsilon$ hold for $j\in[0,\min\{k,m\}-1]$. It follows from \eqref{eqalfnew2} and the definition of $\alpha_{k}^{SD}$ that
\begin{align*}\label{invsd}
\alpha_{k}^{-1}&\geq(\alpha_{k-1}^{SD})^{-1}
=\frac{\sum_{i=1}^n\lambda_i(g^{(i)}_{k-1})^2}
  {\sum_{i=1}^n(g^{(i)}_{k-1})^2}
  \geq\frac{\lambda_{l+1}\sum_{i=l+1}^n(g^{(i)}_{k-1})^2}
{G(k-1,l)+\sum_{i=l+1}^n(g^{(i)}_{k-1})^2}\nonumber\\
&\geq\frac{\lambda_{l+1}}
{\epsilon/2\epsilon+1}\geq\frac{2}{3}\lambda_{l+1},
\end{align*}
where the first inequality in the second line is due to the assumption and $k-1\in\{\max\{k-1,0\},\ldots,k\}$. Thus, (ii) holds for method the \eqref{news2}.
For the method \eqref{news}, we obtain the desired inequality by replacing $k-1$ with $k$ in the above analysis.
This completes the proof.
\end{proof}

\section{Extension to bound constrained optimization}
\label{sec4}
In this section, we would like to extend the strategies and the stepsize (\ref{baralp}) discussed in previous sections for the bound constrained optimization.
\begin{equation}\label{conprob}
\min_{x\in\Omega} ~~f(x),
\end{equation}
where $f$ is a Lipschitz continuously differentiable function defined on the set
$\Omega=\{x\in\mathbb{R}^n|~l\leq x\leq u\}$. Here, $l \le x \le u$ means componentwise
$l_i \le x_i \le u_i$ for all $i=1, \ldots, n$.
Clearly, when $l_i=-\infty$ and $u_i=+\infty$ for all $i$, problem \eqref{conprob} reduces to an unconstrained problem.

Our algorithm belongs to the class of projected gradient methods, which update the iterates as
\[
  x_{k+1}=x_{k}+\lambda_{k}d_{k},
\]
where $\lambda_{k}$ is a step length determined by some line searches and $d_k$ is the search direction given by
\begin{equation}\label{dirc}
  d_k=P_{\Omega}(x_k-\alpha_kg_k)-x_k.
\end{equation}
Here, $P_{\Omega}(\cdot)$ is the Euclidean projection onto $\Omega$ and $\alpha_k$ is our proposed stepsize.

For a general objective function, both $\alpha_{k-1}^{AOPT}$ in (\ref{saopt}) and $\bar{\alpha}_{k-1}$ in (\ref{baralp}) can not be computed as in
Section \ref{sec2} because the Hessian is usually difficult to obtain. Hence, we replace $\alpha_{k-1}^{AOPT}$ by the following positive stepsize suggested by Dai et al. \cite{dai2015positive}:
\begin{equation}\label{pdy0}
  \alpha_{k}^{P}=\frac{\|s_{k-1}\|}{\|y_{k-1}\|}.
\end{equation}
It is easy to see this stepsize is the geometrical mean of the two BB stepsizes in (\ref{sbb})
and it will reduce to $\alpha_{k-1}^{AOPT}$ when the objective function is quadratic.
One may see \cite{dai2015positive} for details about $\alpha_{k}^{P}$.
Note that by applying gradient projection methods for bound constrained optimization,
the variables which are at the boundary usually changes during early
iterations and often become unchanged at the end. So, the algorithm usually eventually solves an
unconstrained problem in the subspace corresponding to free variables. Hence, for bound constrained
optimization, we modify the stepsize (\ref{pdy0}) as the following:
\begin{equation}\label{pdy}
  \bar{\alpha}_{k}^{P}=\frac{\|s_{k-1}\|}{\|\bar{y}_{k-1}\|},
\end{equation}
where
\begin{equation} \label{bary}
  \bar{y}_{k-1}^{(i)}=\left\{
                    \begin{array}{ll}
                      0, & \hbox{if $s_{k-1}^{(i)}=0$;} \\
                      g_k^{(i)}-g_{k-1}^{(i)}, & \hbox{otherwise.}
                    \end{array}
                  \right.
\end{equation}

We now reformulate the stepsize $\bar{\alpha}_{k-1}$ in (\ref{baralp}) for general functions.
In fact, for quadratics, by the update rule \eqref{eqitr} we have
\begin{equation*}
  g_{k-1}^Tg_k=\|g_{k-1}\|^2-\alpha_{k-1}g_{k-1}^TAg_{k-1}=g_{k-1}^TAg_{k-1}(\alpha_{k-1}^{SD}-\alpha_{k-1})
\end{equation*}
and
\begin{equation*}
  g_{k-1}^TAg_k=g_{k-1}^TAg_{k-1}-\alpha_{k-1}g_{k-1}^TA^2g_{k-1}=g_{k-1}^TA^2g_{k-1}(\alpha_{k-1}^{MG}-\alpha_{k-1}),
\end{equation*}
which together with \eqref{baralp}, \eqref{ddk} and \eqref{dadk} give
\begin{equation*}
  \bar{\alpha}_k=\frac{2-2\frac{g_{k-1}^Tg_k}{\|g_{k-1}\|\|g_k\|}}{\frac{1}{\alpha_{k-1}^{SD}}+
  \frac{1}{\alpha_{k}^{SD}} -2\frac{g_{k-1}^TAg_k}{\|g_{k-1}\|\|g_k\|}}
 =\frac{2-2\frac{\|g_{k-1}\|}{\|g_k\|}\frac{1}{\alpha_{k}^{BB1}}(\alpha_{k}^{BB1}-\alpha_{k-1})} {\frac{1}{\alpha_{k}^{BB1}}+
  \frac{1}{\alpha_{k}^{SD}} -2\frac{\|g_{k-1}\|}{\|g_k\|}\frac{1}{\alpha_{k}^{BB1}\alpha_{k}^{BB2}}(\alpha_{k}^{BB2}-\alpha_{k-1})}.
\end{equation*}
Similarly as before, we would modify the BB stepsizes in the above formula, and replace $\alpha_{k}^{BB1}$ and $\alpha_{k}^{BB2}$ by
\begin{equation}\label{bsbb-1-2}
 \bar{\alpha}_k^{BB1}  =\frac{s_{k-1}^Ts_{k-1}}{s_{k-1}^T\bar{y}_{k-1}} \quad \mbox{and} \quad  \bar{\alpha}_k^{BB2}=\frac{s_{k-1}^T\bar{y}_{k-1}}{\bar{y}_{k-1}^T\bar{y}_{k-1}},
\end{equation}
respectively, where $\bar{y}_{k-1}$ is given in (\ref{bary}). Note that in fact $\alpha_k^{BB1}$ automatically
takes care of the changes of free variables since $\alpha_k^{BB1} = \bar{\alpha}_k^{BB1}$.
Then, by replacing the iteration number $k$ with $k-1$, we have
\begin{align}\label{newsunc}
  \bar{\alpha}_{k-1}
 &=\frac{2-2\frac{\|g_{k-2}\|}{\|g_{k-1}\|}\frac{1}{\bar{\alpha}_{k-1}^{BB1}}(\bar{\alpha}_{k-1}^{BB1}-\alpha_{k-2})} {\frac{1}{\bar{\alpha}_{k-1}^{BB1}}+
  \frac{1}{\bar{\alpha}_{k}^{BB1}} -2\frac{\|g_{k-2}\|}{\|g_{k-1}\|}\frac{1}{\bar{\alpha}_{k-1}^{BB1}
  \bar{\alpha}_{k-1}^{BB2}}(\bar{\alpha}_{k-1}^{BB2}-\alpha_{k-2})}.
\end{align}

To ensure global convergence and achieve good performance, nonmonotone line searches \cite{dai2001adaptive,grippo1986nonmonotone,zhang2004nonmonotone} are usually employed for BB-like methods.
Here, we prefer to use the adaptive nonmonotone line search proposed by Dai and Zhang \cite{dai2001adaptive},
which is designed to accept BB stepsizes as frequently as possible. Particularly, the step length
$\lambda_{k} = 1$ is accepted if
\begin{equation} \label{nonmls}
f(x_{k}+ d_{k})\leq f_{r} +\sigma g_{k}^Td_{k},
\end{equation}
where $f_{r}$ is the so-called reference function value adaptively updated by the rules given
in \cite{dai2001adaptive} and  $\sigma\in(0,1)$ is a line search parameter.
However, when (\ref{nonmls}) is not accepted, an Armijo-type back tracking line search is performed to
find the step length $\lambda_{k}$ satisfying a relatively more strict condition
\begin{equation} \label{nonmls-2}
f(x_{k}+\lambda_{k}d_{k})\leq \min\{f_{\max},f_{r}\}+\sigma\lambda_{k}g_{k}^Td_{k},
\end{equation}
where $f_{\max}$ is the maximal function value in recent $M$ iterations, i.e.,
\begin{equation*}
  f_{\max}=\max_{0\leq i\leq \min\{k,M-1\}} f(x_{k-i}).
\end{equation*}
It has been observed that such an nonmonotone line search is specially suitable for BB-like methods \cite{dai2001adaptive}.

Our specific gradient projection algorithm combining with the above nonmonotone line search is stated as Algorithm \ref{al1}.
It is proved in \cite{dai2001adaptive} that when the objective function is Lipschitz continuously differentiable,
Algorithm \ref{al1}  ensures convergence in the sense that
\begin{equation*}
  \lim\inf_{k\rightarrow\infty}\|g_k\|=0.
\end{equation*}

\begin{algorithm}[h]
\caption{Gradient method for bound constrained minimization}\label{al1}
\begin{algorithmic}[1]

\STATE Initialization: $x_{1}\in \mathbb{R}^n$, $\epsilon,\sigma\in(0,1)$, $M,h,s\in\mathbb{N}$, $\alpha_{1}\in[\alpha_{\min},\alpha_{\max}]$.

\WHILE{$\|g_k\|>\epsilon$}

\STATE Compute the search direction $d_k$ by \eqref{dirc};

\STATE Determine $\lambda_{k}$ by nonmonotone line search (\ref{nonmls}) and (\ref{nonmls-2});

\STATE $x_{k+1}=x_{k}+\lambda_{k}d_{k}$;

\IF{$s_{k}^Ty_{k}>0$}

\IF{$\textrm{mod}(k,h+s)\geq h$}

\STATE Compute $\bar{\alpha}_{k}$ by \eqref{newsunc};\\

\IF{$\bar{\alpha}_{k}>0$}

\STATE $\tilde{\alpha}_{k+1} = \min\{\bar{\alpha}_{k},\bar{\alpha}_{k+1}^P\}$;\\

\ELSE

\STATE $\tilde{\alpha}_{k+1} = \bar{\alpha}_{k+1}^{BB2}$;\\

\ENDIF

\ELSE
\STATE $\tilde{\alpha}_{k+1} = \bar{\alpha}_{k+1}^{P}$;\\
\ENDIF

\STATE $\alpha_{k+1} = \max\{\alpha_{\min},\min\{\tilde{\alpha}_{k+1},\alpha_{\max}\}\};$

\ELSE
\STATE $\alpha_{k+1} = 1/\|g_{k+1}\|;$
\ENDIF
\ENDWHILE
\end{algorithmic}
\end{algorithm}


For Algorithm \ref{al1}, we have the following additional comments.
\begin{remark}
When $s_{k}^Ty_{k}\leq 0$, both $\bar{\alpha}_{k+1}^{BB1}$ and $\bar{\alpha}_{k+1}^{BB2}$ are not
well-defined. In this case, we simply take the stepsize $\alpha_{k+1} = 1/\|g_{k+1}\|$.
Moreover, when the stepsize $\bar{\alpha}_{k}$ is negative, we would like to take the shorter
stepsize $\bar{\alpha}_{k+1}^{BB2}$, since $\bar{\alpha}_{k+1}^{BB2}=\min\{\bar{\alpha}_{k+1}^{P},\bar{\alpha}_{k+1}^{BB1},\bar{\alpha}_{k+1}^{BB2}\}$. Here, $0 <\alpha_{\min} << \alpha_{\max}$ serves as the stepsize
safeguards.
\end{remark}

We would also propose two variants of Algorithm \ref{al1}.
As mentioned in Section 2, a small $\bar{\alpha}_{k}$ will be generated if there are
subsequences $\{\alpha_{k_i}\}$ approximating $\frac{1}{\lambda_i}$ for all but
the first a few largest eigenvalues. It has been pointed out in \cite{dai2005analysis,yuan2008step} that the BB method reduces the gradient components more or less at the same asymptotic rate. In other words, the BB stepsize will approximate all the reciprocals of eigenvalues during the iteration process. Similar observations have been presented in \cite{frassoldati2008new}. Thus, for quadratic problem (\ref{eqpro}) we may consider
the following two variants of (\ref{news2}), which combine BB stepsizes with the new stepsize (\ref{baralp}):
\begin{equation}\label{news3}
\alpha_{k}=
\begin{cases}
\alpha_{k}^{BB1},& \text{if $\textrm{mod}(k,h+s)<h$}; \\
\min\{\alpha_{k}^{BB1},\bar{\alpha}_{k-1}\},& \text{otherwise},
\end{cases}
\end{equation}
and
\begin{equation}\label{news4}
\alpha_{k}=
\begin{cases}
\alpha_{k}^{BB2},& \text{if $\textrm{mod}(k,h+s)<h$}; \\
\min\{\alpha_{k}^{BB2},\bar{\alpha}_{k-1}\},& \text{otherwise}.
\end{cases}
\end{equation}
In fact, we have also found reasonably good numerical performances of the methods \eqref{news3} and \eqref{news4}
for minimizing quadratic functions. To generalize the methods \eqref{news3} and \eqref{news4} for
bound constrained optimization, we can replace $\bar{\alpha}_{k+1}^{P}$ in lines 10 and 15 by
$\bar{\alpha}_{k+1}^{BB1}$ and $\bar{\alpha}_{k+1}^{BB2}$, respectively. In what follows, we refer to
Algorithm \ref{al1} using $\bar{\alpha}_{k+1}^{P}$, $\alpha_{k+1}^{BB1}$ and $\bar{\alpha}_{k+1}^{BB2}$
in lines 10 and 15 as A1, A1-BB1 and A1-BB2, respectively.


\section{Numerical results}\label{sec5}
In this section, we do  numerical experiments of the proposed methods
for solving both quadratic and bound constrained optimization problems.
All our codes were written in Matlab.

\subsection{Quadratic problems}
Firstly, we compare our methods \eqref{news} and \eqref{news2} with the DY method \eqref{sdy} in \cite{dai2005analysis}, the ABB$_{\min2}$ method in \cite{frassoldati2008new}, and
the SDC method \eqref{ssdc} in \cite{de2014efficient} for minimizing quadratic problems.
Note that the SDC method has been shown performing better than its monotone variants \cite{de2014efficient}.
For all the comparison methods, the iteration stops when
\begin{equation}\label{eqstop}
  \|g_k\|\leq\epsilon\|g_1\|,
\end{equation}
where $\epsilon > 0$ is a given tolerance, or the iteration number exceeds 20,000.
Based on the observation from Figures \ref{appstep1} and \ref{appstep2}, we tested
$h$ with values 10 and 20 for our methods.
As in \cite{frassoldati2008new}, the parameter $\tau$ of the ABB$_{\min2}$ method was set to 0.9 for all the problems.

Our first set of test problems are quadratic problems (\ref{eqpro})
from \cite{dhl2018,dai2003altermin,friedlander1998gradient,zhou2006gradient},
whose Hessian have different spectral distributions.
In particular,  the objective function has Hessian $A=QVQ^T$ with
\begin{equation*}
  Q=(I-2w_3w_3^T)(I-2w_2w_2^T)(I-2w_1w_1^T),
\end{equation*}
where $w_1$, $w_2$, and $w_3$ are unitary random vectors,  $V=\textrm{diag}(v_1,\ldots,v_n)$ is a diagonal matrix with
 $v_1=1$ and $v_n=\kappa$, and $v_j$, $j=2,\ldots,n-1$, being  randomly generated between 1 and $\kappa$.
 The vector $b$ were randomly generated with components between $-10$ and $10$.
 Five sets of different spectral distributions of the test problems
 are given in Table \ref{tbspe} and the problem dimension is set as $n=1000$.
 For each problem set, three different values of condition number $\kappa$
 and tolerances $\epsilon$ are tested. For each value of $\kappa$ or $\epsilon$, $10$ problem instances
 were randomly generated. Tables \ref{tbrandp} and \ref{tbrandp2} show the average number of iterations over those instances
 with the starting point $x_1=(1,\ldots,1)^T$, where the parameter pair $(h,s)$ used for the SDC method was set to $(8,6)$ which is more efficient than other choices for this test set.

\begin{table}[h]
\tbl{Distributions of $v_j$.}
{\begin{tabular}{|c|c|}
\hline
 \multirow{1}{*}{Problem} &\multicolumn{1}{c|}{Spectrum} \\
\hline
\multirow{1}{*}{1} &$\{v_2,\ldots,v_{n-1}\}\subset(1,\kappa)$	\\
\hline
 \multirow{2}{*}{2}
&$\{v_2,\ldots,v_{n/5}\}\subset(1,100)$	\\
&$\{v_{n/5+1},\ldots,v_{n-1}\}\subset(\frac{\kappa}{2},\kappa)$	\\
\hline
\multirow{2}{*}{3}
&$\{v_2,\ldots,v_{n/2}\}\subset(1,100)$	\\
&$\{v_{n/2+1},\ldots,v_{n-1}\}\subset(\frac{\kappa}{2},\kappa)$	\\
\hline
\multirow{2}{*}{4}
&$\{v_2,\ldots,v_{4n/5}\}\subset(1,100)$	\\
&$\{v_{4n/5+1},\ldots,v_{n-1}\}\subset(\frac{\kappa}{2},\kappa)$	\\
\hline
\multirow{3}{*}{5}
&$\{v_2,\ldots,v_{n/5}\}\subset(1,100)$	\\
&$\{v_{n/5+1},\ldots,v_{4n/5}\}\subset(100,\frac{\kappa}{2})$	\\
&$\{v_{4n/5+1},\ldots,v_{n-1}\}\subset(\frac{\kappa}{2},\kappa)$	\\
\hline
\end{tabular}}
\label{tbspe}
\end{table}

We can see from Table \ref{tbrandp} that, our method \eqref{news} is  competitive with the DY, ABB$_{\min 2}$
and SDC methods. For a fixed $h$, larger values of $s$ seem to be preferable for the method \eqref{news}. In addition, different settings of $s$ lead to comparable results, with differences of less than 10\% in the number of iterations for most of the test problems. For the first problem set, our method \eqref{news} outperforms the DY and SDC methods, although the ABB$_{\min 2}$ method seems surprisingly efficient for this first problem set among the compared methods. Particularly, when a high accuracy is required, the method \eqref{news} with $(h,s)=(20,100)$ often
takes less than $\frac{1}{6}$ and $\frac{1}{4}$  number of iterations needed by the DY and SDC methods, respectively. As for the second to fourth problem sets, the method \eqref{news} with different settings performs better than the DY and ABB$_{\min 2}$ methods and also peforms better than the SDC method if proper $h$ and $s$ are selected.
\begin{table}[H]
\setlength{\tabcolsep}{0.2ex}
\tbl{Number of averaged iterations of the method \eqref{news}, DY, ABB$_{\min 2}$ and SDC on problems in Table \ref{tbspe}.}
{\begin{scriptsize}
\begin{tabular}{|c|c|c|c|c|c|c|c|c|c|c|c|c|c|c|}
\hline
 \multicolumn{1}{|c|}{\multirow{2}{*}{problem}} &\multicolumn{1}{c|}{\multirow{2}{*}{$\epsilon$}}
 &\multicolumn{10}{c|}{$(h,s)$ for the method \eqref{news}} &\multirow{2}{*}{DY} &\multirow{2}{*}{ABB$_{\min 2}$} &\multirow{2}{*}{SDC}\\
\cline{3-12}
 \multicolumn{1}{|c|}{}& \multicolumn{1}{c|}{}   &$(10,20)$ &$(10,30)$ &$(10,50)$ &$(10,80)$  &$(10,100)$  &$(20,20)$ &$(20,30)$ &$(20,50)$ &$(20,80)$  &$(20,100)$
  & & &\\
\hline
\multicolumn{1}{|c|}{\multirow{3}{*}{1}}
&\multicolumn{1}{c|}{\multirow{1}{*}{$10^{-6}$}}	 &332.9  &317.7  &317.1  &327.6  &340.1   &360.3  &326.6  &321.8  &324.0  &327.7             &\multirow{1}{*}{439.0} &\multirow{1}{*}{\textbf{249.3}} &\multirow{1}{*}{382.7}   \\
&\multicolumn{1}{c|}{\multirow{1}{*}{$10^{-9}$}}	 &2325.4 &1298.7 &1176.3 &875.2  &931.3   &1509.2 &1326.9 &1030.9 &789.6  &754.9     &\multirow{1}{*}{3979.5} &\multirow{1}{*}{\textbf{489.1}} &\multirow{1}{*}{2970.0} \\
&\multicolumn{1}{c|}{\multirow{1}{*}{$10^{-12}$}}  &4332.8 &2379.9 &2028.7 &1349.0 &1345.6  &2795.3 &2114.3 &1651.0 &1291.8 &1111.0 &\multirow{1}{*}{7419.6} &\multirow{1}{*}{\textbf{629.6}} &\multirow{1}{*}{5113.0} \\
\hline

\multicolumn{1}{|c|}{\multirow{3}{*}{2}}
&\multicolumn{1}{c|}{\multirow{1}{*}{$10^{-6}$}}   &243.2  &227.8  &236.2  &260.9  &281.5   &278.4  &243.9  &\textbf{226.9}  &237.0  &244.2              &\multirow{1}{*}{342.5} &\multirow{1}{*}{394.0} &\multirow{1}{*}{228.5}   \\
&\multicolumn{1}{c|}{\multirow{1}{*}{$10^{-9}$}}   &895.7  &823.7  &845.7  &906.7  &966.3   &873.8  &\textbf{803.0}  &825.7  &869.3  &854.6            &\multirow{1}{*}{1584.3} &\multirow{1}{*}{1504.7} &\multirow{1}{*}{891.9}  \\
&\multicolumn{1}{c|}{\multirow{1}{*}{$10^{-12}$}}  &1537.0 &1387.5 &1372.0 &1390.5 &1446.0  &1419.8 &1284.3 &\textbf{1257.4} &1359.0 &1359.3 &\multirow{1}{*}{2748.1} &\multirow{1}{*}{2362.4} &\multirow{1}{*}{1410.0} \\
\hline

\multicolumn{1}{|c|}{\multirow{3}{*}{3}}
&\multicolumn{1}{c|}{\multirow{1}{*}{$10^{-6}$}}   &315.0  &291.4  &315.1  &333.6  &357.2   &345.4 &303.3 &\textbf{272.5} &306.5 &311.4              &\multirow{1}{*}{473.4} &\multirow{1}{*}{471.4} &\multirow{1}{*}{293.5}   \\
&\multicolumn{1}{c|}{\multirow{1}{*}{$10^{-9}$}}   &977.9  &938.0  &938.3  &970.4  &1025.0  &1005.8 &869.2 &870.0 &910.2 &998.7         &\multirow{1}{*}{1766.6} &\multirow{1}{*}{1674.4} &\multirow{1}{*}{\textbf{860.8}}  \\
&\multicolumn{1}{c|}{\multirow{1}{*}{$10^{-12}$}}  &1577.8 &1510.6 &1515.0 &1556.8 &1587.9  &1636.8 &1363.8 &\textbf{1351.3} &1434.6 &1529.0 &\multirow{1}{*}{2846.7} &\multirow{1}{*}{2533.5} &\multirow{1}{*}{1397.7} \\
\hline

\multicolumn{1}{|c|}{\multirow{3}{*}{4}}
&\multicolumn{1}{c|}{\multirow{1}{*}{$10^{-6}$}}   &366.1  &346.9  &393.1  &393.6  &413.0   &405.8 &360.0 &\textbf{334.7} &382.0 &375.9              &\multirow{1}{*}{585.0} &\multirow{1}{*}{670.4} &\multirow{1}{*}{369.2}   \\
&\multicolumn{1}{c|}{\multirow{1}{*}{$10^{-9}$}}   &1069.5 &988.7  &996.2  &1039.7 &1063.1  &1077.5 &\textbf{904.1} &954.3 &982.5 &1029.1      &\multirow{1}{*}{1991.4} &\multirow{1}{*}{1644.8} &\multirow{1}{*}{977.2}  \\
&\multicolumn{1}{c|}{\multirow{1}{*}{$10^{-12}$}}  &1658.8 &1508.7 &1517.8 &1620.1 &1656.5  &1658.3 &1421.0 &\textbf{1392.1} &1492.0 &1565.5 &\multirow{1}{*}{3072.5} &\multirow{1}{*}{2559.2} &\multirow{1}{*}{1484.2} \\
\hline

\multicolumn{1}{|c|}{\multirow{3}{*}{5}}
&\multicolumn{1}{c|}{\multirow{1}{*}{$10^{-6}$}}   &826.3  &\textbf{797.6}  &845.9  &856.3  &909.4   &895.5 &876.1 &885.7 &877.7 &899.1            &\multirow{1}{*}{878.8} &\multirow{1}{*}{1041.9} &\multirow{1}{*}{934.2}  \\
&\multicolumn{1}{c|}{\multirow{1}{*}{$10^{-9}$}}   &3500.9 &3427.6 &3388.3 &3374.9 &3432.8  &3383.2 &3466.1 &3225.5 &3234.4 &\textbf{3156.2} &\multirow{1}{*}{4275.8} &\multirow{1}{*}{3293.9} &\multirow{1}{*}{4117.5} \\
&\multicolumn{1}{c|}{\multirow{1}{*}{$10^{-12}$}}  &5667.2 &5428.9 &5424.8 &5303.5 &5540.9  &5868.5 &5956.5 &5484.2 &5228.4 &\textbf{5153.8} &\multirow{1}{*}{7661.2} &\multirow{1}{*}{5310.5} &\multirow{1}{*}{6465.7} \\
\hline

\multicolumn{1}{|c|}{\multirow{3}{*}{total}}
&\multicolumn{1}{c|}{\multirow{1}{*}{$10^{-6}$}} &2083.5 &\textbf{1981.4} &2107.4 &2172.0 &2301.2 &2285.4 &2109.9 &2041.6 &2127.2 &2158.3 &2718.7 &2827.0 &2208.1                 \\
&\multicolumn{1}{c|}{\multirow{1}{*}{$10^{-9}$}}  &8769.4 &7476.7 &7344.8 &7166.9 &7418.5 &7849.5 &7369.3 &6906.4 &\textbf{6786.0} &6793.5 &13597.6 &8606.9 &9817.4               \\
&\multicolumn{1}{c|}{\multirow{1}{*}{$10^{-12}$}} &14773.6 &12215.6 &11858.3 &11219.9 &11576.9 &13378.7 &12139.9 &11136.0 &10805.8 &\textbf{10718.6} &23748.1 &13395.2 &15870.6   \\
\hline

\end{tabular}
\end{scriptsize}
}\label{tbrandp}
\end{table}

\begin{table}[h]
\setlength{\tabcolsep}{0.2ex}
\tbl{Number of averaged iterations of the method \eqref{news2}, DY, ABB$_{\min 2}$ and SDC on problems in Table \ref{tbspe}.}
{
\begin{scriptsize}
\begin{tabular}{|c|c|c|c|c|c|c|c|c|c|c|c|c|c|c|}
\hline
 \multicolumn{1}{|c|}{\multirow{2}{*}{problem}} &\multicolumn{1}{c|}{\multirow{2}{*}{$\epsilon$}}
 &\multicolumn{10}{c|}{$(h,s)$ for the method \eqref{news2}} &\multirow{2}{*}{DY} &\multirow{2}{*}{ABB$_{\min 2}$} &\multirow{2}{*}{SDC}\\
\cline{3-12}
 \multicolumn{1}{|c|}{}& \multicolumn{1}{c|}{}   &$(10,20)$ &$(10,30)$ &$(10,50)$ &$(10,80)$  &$(10,100)$  &$(20,20)$ &$(20,30)$ &$(20,50)$ &$(20,80)$  &$(20,100)$
  & & &\\
  \hline
\multicolumn{1}{|c|}{\multirow{3}{*}{1}}
&\multicolumn{1}{c|}{\multirow{1}{*}{$10^{-6}$}}	&345.2  &338.9  &351.7  &330.9  &344.4  &354.6 &335.1 &329.0 &336.2 &336.0             &\multirow{1}{*}{439.0} &\multirow{1}{*}{\textbf{249.3}} &\multirow{1}{*}{382.7}   \\
&\multicolumn{1}{c|}{\multirow{1}{*}{$10^{-9}$}}	&1833.5 &1537.8 &1408.0 &989.7  &826.0   &1602.5 &1147.4 &966.2 &932.7 &812.8      &\multirow{1}{*}{3979.5} &\multirow{1}{*}{\textbf{489.1}} &\multirow{1}{*}{2970.0} \\
&\multicolumn{1}{c|}{\multirow{1}{*}{$10^{-12}$}} &3569.5 &2340.2 &2249.1 &1405.7 &1124.8  &2293.7 &2169.0 &1630.1 &1304.4 &1068.6 &\multirow{1}{*}{7419.6} &\multirow{1}{*}{\textbf{629.6}} &\multirow{1}{*}{5113.0} \\
\hline

\multicolumn{1}{|c|}{\multirow{3}{*}{2}}
&\multicolumn{1}{c|}{\multirow{1}{*}{$10^{-6}$}}  &220.4  &\textbf{218.7}  &222.7  &224.9  &234.4  &278.0 &263.2 &244.6 &242.6 &255.8              &\multirow{1}{*}{342.5} &\multirow{1}{*}{394.0} &\multirow{1}{*}{228.5}   \\
&\multicolumn{1}{c|}{\multirow{1}{*}{$10^{-9}$}}  &821.9  &815.7  &760.9  &\textbf{740.7}  &807.4   &956.5 &910.4 &828.9 &818.4 &862.3            &\multirow{1}{*}{1584.3} &\multirow{1}{*}{1504.7} &\multirow{1}{*}{891.9}  \\
&\multicolumn{1}{c|}{\multirow{1}{*}{$10^{-12}$}} &1277.2 &1247.4 &\textbf{1214.4} &1216.1 &1216.3  &1544.0 &1466.5 &1344.9 &1307.5 &1364.2 &\multirow{1}{*}{2748.1} &\multirow{1}{*}{2362.4} &\multirow{1}{*}{1410.0} \\
\hline

\multicolumn{1}{|c|}{\multirow{3}{*}{3}}
&\multicolumn{1}{c|}{\multirow{1}{*}{$10^{-6}$}}  &\textbf{272.8}  &275.9  &284.6  &291.7  &309.0  &353.6 &333.0 &305.4 &303.8 &320.3              &\multirow{1}{*}{473.4} &\multirow{1}{*}{471.4} &\multirow{1}{*}{293.5}   \\
&\multicolumn{1}{c|}{\multirow{1}{*}{$10^{-9}$}}  &\textbf{842.9}  &848.6  &844.4  &889.3  &863.5   &1028.9 &964.7 &946.7 &929.0 &938.3           &\multirow{1}{*}{1766.6} &\multirow{1}{*}{1674.4} &\multirow{1}{*}{860.8}  \\
&\multicolumn{1}{c|}{\multirow{1}{*}{$10^{-12}$}} &1367.1 &1356.7 &1312.8 &1350.1 &\textbf{1305.4}  &1644.4 &1511.1 &1472.4 &1446.4 &1368.2 &\multirow{1}{*}{2846.7} &\multirow{1}{*}{2533.5} &\multirow{1}{*}{1397.7} \\
\hline

\multicolumn{1}{|c|}{\multirow{3}{*}{4}}
&\multicolumn{1}{c|}{\multirow{1}{*}{$10^{-6}$}}  &337.3  &348.8  &363.8  &344.2  &\textbf{333.9}  &413.9 &395.4 &400.2 &388.6 &392.4             &\multirow{1}{*}{585.0} &\multirow{1}{*}{670.4} &\multirow{1}{*}{369.2}   \\
&\multicolumn{1}{c|}{\multirow{1}{*}{$10^{-9}$}}  &894.3  &895.2  &\textbf{869.1}  &872.3  &874.8   &1101.8 &1029.6 &1014.8 &989.3 &977.1        &\multirow{1}{*}{1991.4} &\multirow{1}{*}{1644.8} &\multirow{1}{*}{977.2}  \\
&\multicolumn{1}{c|}{\multirow{1}{*}{$10^{-12}$}} &1389.2 &1385.7 &1364.3 &\textbf{1344.2} &1393.1  &1676.0 &1591.9 &1554.4 &1440.7 &1438.8&\multirow{1}{*}{3072.5} &\multirow{1}{*}{2559.2} &\multirow{1}{*}{1484.2} \\
\hline

\multicolumn{1}{|c|}{\multirow{3}{*}{5}}
&\multicolumn{1}{c|}{\multirow{1}{*}{$10^{-6}$}}  &\textbf{794.0}  &802.7  &811.9  &831.9  &876.0  &843.1 &832.3 &822.3 &805.9 &853.5              &\multirow{1}{*}{878.8} &\multirow{1}{*}{1041.9} &\multirow{1}{*}{934.2}  \\
&\multicolumn{1}{c|}{\multirow{1}{*}{$10^{-9}$}}  &3415.2 &3252.1 &3081.2 &3262.6 &2995.6  &3150.4 &3259.6 &2990.9 &\textbf{2980.3} &3045.4 &\multirow{1}{*}{4275.8} &\multirow{1}{*}{3293.9} &\multirow{1}{*}{4117.5} \\
&\multicolumn{1}{c|}{\multirow{1}{*}{$10^{-12}$}} &5492.1 &5272.7 &5102.2 &5102.0 &4982.8  &4861.3 &5150.9 &\textbf{4700.0} &5035.7 &4808.7 &\multirow{1}{*}{7661.2} &\multirow{1}{*}{5310.5} &\multirow{1}{*}{6465.7} \\
\hline

\multicolumn{1}{|c|}{\multirow{3}{*}{total}}
&\multicolumn{1}{c|}{\multirow{1}{*}{$10^{-6}$}}  &\textbf{1969.7} &1985.0 &2034.7 &2023.6 &2097.7 &2243.2 &2159.0 &2101.5 &2077.1 &2158.0 &2718.7 &2827.0 &2208.1                  \\
&\multicolumn{1}{c|}{\multirow{1}{*}{$10^{-9}$}}   &7807.8 &7349.4 &6963.6 &6754.6 &\textbf{6367.3} &7840.1 &7311.7 &6747.5 &6649.7 &6635.9 &13597.6 &8606.9 &9817.4               \\
&\multicolumn{1}{c|}{\multirow{1}{*}{$10^{-12}$}}  &13095.1 &11602.7 &11242.8 &10418.1 &\textbf{10022.4} &12019.4 &11889.4 &10701.8 &10534.7 &10048.5 &23748.1 &13395.2 &15870.6   \\
\hline

\end{tabular}
\end{scriptsize}
}\label{tbrandp2}
\end{table}

\noindent The method \eqref{news} also performs better than the DY and SDC methods on the last problem set.
From the total number of iterations, we can see the overall performance of the method \eqref{news} is quite good.
Here, we want to point out that our method \eqref{news} and the DY method are monotone, while the ABB$_{\min 2}$
and SDC methods are not.

Table \ref{tbrandp2} shows the averaged number of iterations of our method \eqref{news2} for the first set of test problems. For comparison purposes, the results of the DY, ABB$_{\min 2}$ and SDC methods
are also listed here. Similar performance as the method \eqref{news} can be observed. In particular, for each accuracy level, our method \eqref{news2} takes around 30\% less total iterations than the ABB$_{\min 2}$ method and also much less total iterations than the DY and SDC methods. Notice that problems of the last set are difficult for the compared methods since more iterations are needed than other four sets. However, the method \eqref{news2} always dominates the compared three methods except with the pair $(h,s)=(10,20)$. As compared with the method \eqref{news}, the retard strategy used in the method \eqref{news2} tends to improve the performance when $h=10$. For the case $h=20$, the method \eqref{news2} is also comparable to and better than \eqref{news} in terms of total number of iterations.

Our second set of quadratic test problems are the
two large-scale real problems Laplace1(a) and Laplace1(b) described in \cite{fletcher2005barzilai}. Both of the problems require the solution of a system of linear equations derived from a 3D Laplacian on a box, discretized using a standard 7-point finite difference stencil. The solution is fixed by a Gaussian function whose center is $(\alpha,\beta,\gamma)$, multiplied by $x(x-1)$. A parameter $\sigma$ is used to control the rate of decay of the Gaussian. Both Laplace1(a) and Laplace1(b) have $n=N^3$ variables, where $N$ is the interior
nodes taken in each coordinate direction, and have a highly sparse Hessian matrix with condition number
 $10^{3.61}$. We refer the readers to \cite{fletcher2005barzilai} for more details on these problems. In our tests, the associated parameters are set as follows:
\begin{align*}
&(a)~~ \sigma=20,~ \alpha=\beta=\gamma=0.5;\\
 &(b)~~ \sigma=50,~ \alpha=0.4,~ \beta=0.7,~ \gamma=0.5.
\end{align*}
We use the null vector as the starting point. The number of iterations required by the compared methods
for solving the two problems Laplace1(a) and Laplace1(b) are listed in Tables \ref{tbnLap} and \ref{tbnLap2}, respectively.

\begin{table}[h]
\setlength{\tabcolsep}{0.2ex}
\tbl{Number of iterations of the method \eqref{news}, DY, ABB$_{\min 2}$ and SDC on the 3D Laplacian problem.}
{
\begin{scriptsize}
\begin{tabular}{|c|c|c|c|c|c|c|c|c|c|c|c|c|c|c|}
\hline
 \multicolumn{1}{|c|}{\multirow{2}{*}{$n$}} &\multicolumn{1}{c|}{\multirow{2}{*}{$\epsilon$}}
 &\multicolumn{10}{c|}{$(h,s)$ for the method \eqref{news}} &\multirow{2}{*}{DY} &\multirow{2}{*}{ABB$_{\min 2}$} &\multirow{2}{*}{SDC}\\
\cline{3-12}
 \multicolumn{1}{|c|}{}& \multicolumn{1}{c|}{}  &$(10,20)$ &$(10,30)$ &$(10,50)$ &$(10,80)$  &$(10,100)$  &$(20,20)$ &$(20,30)$ &$(20,50)$ &$(20,80)$  &$(20,100)$
  & & &\\
 \hline
\multicolumn{15}{|c|}{Problem Laplace1(a)}\\
 \hline

 \multicolumn{1}{|c|}{\multirow{3}{*}{$60^{3}$}}
&\multicolumn{1}{c|}{\multirow{1}{*}{$10^{-6}$}}    &271  &241  &243  &197  &331   &321  &247  &211  &211  &241  &249  &\textbf{192}  &213   \\
&\multicolumn{1}{c|}{\multirow{1}{*}{$10^{-9}$}}    &348  &\textbf{257}  &421  &357  &334   &521  &351  &281  &303  &331  &373  &329  &393   \\
&\multicolumn{1}{c|}{\multirow{1}{*}{$10^{-12}$}}   &451  &394  &437  &452  &441   &523  &401  &\textbf{351}  &459  &482  &546  &401  &529   \\
 \hline

 \multicolumn{1}{|c|}{\multirow{3}{*}{$80^{3}$}}
&\multicolumn{1}{c|}{\multirow{1}{*}{$10^{-6}$}}    &315  &441  &362  &303  &331   &395  &426  &351  &301  &361  &383  &\textbf{289}  &297  \\
&\multicolumn{1}{c|}{\multirow{1}{*}{$10^{-9}$}}    &436  &480  &481  &510  &\textbf{349}   &521  &525  &420  &402  &407  &570  &430  &553  \\
&\multicolumn{1}{c|}{\multirow{1}{*}{$10^{-12}$}}   &602  &548  &601  &631  &\textbf{451}   &762  &677  &561  &526  &601  &789  &608  &705  \\
 \hline

 \multicolumn{1}{|c|}{\multirow{3}{*}{$100^{3}$}}
&\multicolumn{1}{c|}{\multirow{1}{*}{$10^{-6}$}}    &500  &482  &\textbf{301}  &371  &441   &441  &351  &421  &401  &361   &427  &351  &513   \\
&\multicolumn{1}{c|}{\multirow{1}{*}{$10^{-9}$}}    &691  &639  &541  &527  &565   &562  &\textbf{453}  &556  &457  &601   &651  &485  &609   \\
&\multicolumn{1}{c|}{\multirow{1}{*}{$10^{-12}$}}    &826  &900  &649  &808  &771   &881  &734  &701  &\textbf{602}  &721  &918  &687  &825   \\
\hline

 \multicolumn{1}{|c|}{\multirow{3}{*}{total}}
&\multicolumn{1}{c|}{\multirow{1}{*}{$10^{-6}$}}    &1086  &1164  &906   &871   &1103   &1157  &1024  &983  &913  &963  &1059  &\textbf{832}  &1023        \\
&\multicolumn{1}{c|}{\multirow{1}{*}{$10^{-9}$}}    &1475  &1376  &1443  &1394  &1248   &1604  &1329  &1257  &\textbf{1162}  &1339  &1594  &1244  &1555  \\
&\multicolumn{1}{c|}{\multirow{1}{*}{$10^{-12}$}}   &1879  &1842  &1687  &1891  &1663   &2166  &1812  &1613  &\textbf{1587}  &1804  &2253  &1696  &2059  \\
\hline

\multicolumn{15}{|c|}{Problem Laplace1(b)}\\
 \hline

\multicolumn{1}{|c|}{\multirow{3}{*}{$60^{3}$}}
&\multicolumn{1}{c|}{\multirow{1}{*}{$10^{-6}$}}    &327  &241  &241  &271  &331    &281  &351  &\textbf{211}  &245  &241    &236  &217  &213   \\
&\multicolumn{1}{c|}{\multirow{1}{*}{$10^{-9}$}}    &361  &361  &361  &\textbf{320}  &367    &521  &351  &352  &401  &361    &399  &365  &437   \\
&\multicolumn{1}{c|}{\multirow{1}{*}{$10^{-12}$}}   &509  &482  &481  &451  &\textbf{442}    &641  &451  &585  &502  &507    &532  &502  &555   \\
 \hline

 \multicolumn{1}{|c|}{\multirow{3}{*}{$80^{3}$}}
&\multicolumn{1}{c|}{\multirow{1}{*}{$10^{-6}$}}     &319  &361  &421  &361  &387   &321  &351  &\textbf{281}  &401  &295    &454  &294  &309   \\
&\multicolumn{1}{c|}{\multirow{1}{*}{$10^{-9}$}}     &511  &561  &468  &448  &551   &549  &477  &561  &502  &506    &567  &\textbf{433}  &485   \\
&\multicolumn{1}{c|}{\multirow{1}{*}{$10^{-12}$}}    &751  &702  &653  &658  &652   &801  &638  &739  &\textbf{601}  &669    &794  &634  &766   \\
 \hline

 \multicolumn{1}{|c|}{\multirow{3}{*}{$100^{3}$}}
&\multicolumn{1}{c|}{\multirow{1}{*}{$10^{-6}$}}    &393  &401  &396  &361  &532  &402  &393  &421  &425  &\textbf{361}       &371  &369  &379  \\
&\multicolumn{1}{c|}{\multirow{1}{*}{$10^{-9}$}}    &632  &635  &602  &631  &662  &801  &752  &701  &701  &707       &700  &\textbf{585}  &653  \\
&\multicolumn{1}{c|}{\multirow{1}{*}{$10^{-12}$}}   &931  &961  &902  &901  &991   &961  &1001  &937  &\textbf{802}  &1024    &1038  &880  &965   \\
\hline

 \multicolumn{1}{|c|}{\multirow{3}{*}{total}}
&\multicolumn{1}{c|}{\multirow{1}{*}{$10^{-6}$}}   &1039  &1003  &1058  &993   &1250   &1004  &1095  &913  &1071  &897  &1061  &\textbf{880}  &901        \\
&\multicolumn{1}{c|}{\multirow{1}{*}{$10^{-9}$}}   &1504  &1557  &1431  &1399  &1580   &1871  &1580  &1614  &1604  &1574  &1666  &\textbf{1383}  &1575   \\
&\multicolumn{1}{c|}{\multirow{1}{*}{$10^{-12}$}}  &2191  &2145  &2036  &2010  &2085   &2403  &2090  &2261  &\textbf{1905}  &2200  &2364  &2016  &2286   \\
\hline
\end{tabular}
\end{scriptsize}
}\label{tbnLap}
\end{table}

\begin{table}[h]
\setlength{\tabcolsep}{0.2ex}
\tbl{Number of iterations of the method \eqref{news2}, DY, ABB$_{\min 2}$ and SDC on the 3D Laplacian problem.}
{
\begin{scriptsize}
\begin{tabular}{|c|c|c|c|c|c|c|c|c|c|c|c|c|c|c|}
\hline
 \multicolumn{1}{|c|}{\multirow{2}{*}{$n$}} &\multicolumn{1}{c|}{\multirow{2}{*}{$\epsilon$}}
 &\multicolumn{10}{c|}{$(h,s)$ for the method \eqref{news2}} &\multirow{2}{*}{DY} &\multirow{2}{*}{ABB$_{\min 2}$} &\multirow{2}{*}{SDC}\\
\cline{3-12}
 \multicolumn{1}{|c|}{}& \multicolumn{1}{c|}{}  &$(10,20)$ &$(10,30)$ &$(10,50)$ &$(10,80)$  &$(10,100)$  &$(20,20)$ &$(20,30)$ &$(20,50)$ &$(20,80)$  &$(20,100)$
  & & &\\
 \hline
\multicolumn{15}{|c|}{Problem Laplace1(a)}\\
 \hline

 \multicolumn{1}{|c|}{\multirow{3}{*}{$60^{3}$}}
&\multicolumn{1}{c|}{\multirow{1}{*}{$10^{-6}$}}    &279  &208  &241  &209  &283   &241  &201  &212  &247  &241   &249  &\textbf{192}  &213   \\
&\multicolumn{1}{c|}{\multirow{1}{*}{$10^{-9}$}}    &421  &361  &308  &326  &331   &320  &301  &357  &301  &\textbf{241}   &373  &329  &393   \\
&\multicolumn{1}{c|}{\multirow{1}{*}{$10^{-12}$}}   &486  &424  &342  &380  &361   &363  &396  &427  &\textbf{322}  &367   &546  &401  &529   \\
 \hline

 \multicolumn{1}{|c|}{\multirow{3}{*}{$80^{3}$}}
&\multicolumn{1}{c|}{\multirow{1}{*}{$10^{-6}$}}    &301  &335  &\textbf{271}  &290  &303   &303  &278  &367  &308  &361   &383  &289  &297  \\
&\multicolumn{1}{c|}{\multirow{1}{*}{$10^{-9}$}}    &489  &490  &421  &376  &450   &\textbf{359}  &384  &561  &499  &481   &570  &430  &553  \\
&\multicolumn{1}{c|}{\multirow{1}{*}{$10^{-12}$}}   &676  &681  &552  &596  &551   &524  &585  &636  &515  &\textbf{495}   &789  &608  &705  \\
 \hline

 \multicolumn{1}{|c|}{\multirow{3}{*}{$100^{3}$}}
&\multicolumn{1}{c|}{\multirow{1}{*}{$10^{-6}$}}    &474  &361  &361  &541  &397   &441  &451  &421  &416  &361    &427  &\textbf{351}  &513   \\
&\multicolumn{1}{c|}{\multirow{1}{*}{$10^{-9}$}}    &721  &441  &\textbf{439}  &576  &561   &681  &536  &529  &601  &481    &651  &485  &609   \\
&\multicolumn{1}{c|}{\multirow{1}{*}{$10^{-12}$}}    &811  &761  &612  &740  &771   &786  &701  &644  &756  &\textbf{601}   &918  &687  &825   \\
\hline

 \multicolumn{1}{|c|}{\multirow{3}{*}{total}}
&\multicolumn{1}{c|}{\multirow{1}{*}{$10^{-6}$}}    &1054  &904  &873  &1040  &983  &985  &930  &1000  &971  &963  &1059  &\textbf{832}  &1023           \\
&\multicolumn{1}{c|}{\multirow{1}{*}{$10^{-9}$}}    &1631  &1292  &\textbf{1168}  &1278  &1342  &1360  &1221  &1447  &1401  &1203  &1594  &1244  &1555   \\
&\multicolumn{1}{c|}{\multirow{1}{*}{$10^{-12}$}}   &1973  &1866  &1506  &1716  &1683  &1673  &1682  &1707  &1593  &\textbf{1463}  &2253  &1696  &2059   \\
\hline

\multicolumn{15}{|c|}{Problem Laplace1(b)}\\
 \hline

\multicolumn{1}{|c|}{\multirow{3}{*}{$60^{3}$}}
&\multicolumn{1}{c|}{\multirow{1}{*}{$10^{-6}$}}     &229  &241  &241  &271  &258   &249  &277  &281  &305  &241   &236  &217  &\textbf{213}   \\
&\multicolumn{1}{c|}{\multirow{1}{*}{$10^{-9}$}}     &364  &361  &385  &335  &344   &396  &411  &\textbf{334}  &399  &368   &399  &365  &437   \\
&\multicolumn{1}{c|}{\multirow{1}{*}{$10^{-12}$}}    &546  &451  &541  &467  &551   &561  &443  &\textbf{413}  &505  &496   &532  &502  &555   \\
 \hline

 \multicolumn{1}{|c|}{\multirow{3}{*}{$80^{3}$}}
&\multicolumn{1}{c|}{\multirow{1}{*}{$10^{-6}$}}      &309  &278  &312  &\textbf{271}  &361   &292  &289  &363  &363  &460    &454  &294  &309   \\
&\multicolumn{1}{c|}{\multirow{1}{*}{$10^{-9}$}}      &521  &521  &547  &528  &551   &443  &507  &491  &501  &601    &567  &\textbf{433}  &485   \\
&\multicolumn{1}{c|}{\multirow{1}{*}{$10^{-12}$}}     &684  &692  &721  &619  &\textbf{614}   &681  &642  &631  &679  &721    &794  &634  &766   \\
 \hline

 \multicolumn{1}{|c|}{\multirow{3}{*}{$100^{3}$}}
&\multicolumn{1}{c|}{\multirow{1}{*}{$10^{-6}$}}     &391  &450  &355  &361  &402  &361  &385  &\textbf{351}  &362  &362   &371  &369  &379  \\
&\multicolumn{1}{c|}{\multirow{1}{*}{$10^{-9}$}}     &565  &657  &630  &707  &771  &561  &644  &\textbf{491}  &701  &650   &700  &585  &653  \\
&\multicolumn{1}{c|}{\multirow{1}{*}{$10^{-12}$}}    &\textbf{731}  &904  &972  &921  &771  &880  &911  &841  &814  &881   &1038  &880  &965   \\
\hline

 \multicolumn{1}{|c|}{\multirow{3}{*}{total}}
&\multicolumn{1}{c|}{\multirow{1}{*}{$10^{-6}$}}   &929  &969  &908  &903  &1021  &902  &951  &995  &1030  &1063  &1061  &\textbf{880}  &901              \\
&\multicolumn{1}{c|}{\multirow{1}{*}{$10^{-9}$}}    &1450  &1539  &1562  &1570  &1666  &1400  &1562  &\textbf{1316}  &1601  &1619  &1666  &1383  &1575    \\
&\multicolumn{1}{c|}{\multirow{1}{*}{$10^{-12}$}}   &1961  &2047  &2234  &2007  &1936  &2122  &1996  &\textbf{1885}  &1998  &2098  &2364  &2016  &2286    \\
\hline
\end{tabular}
\end{scriptsize}
}\label{tbnLap2}
\end{table}

From Table \ref{tbnLap} we can see that, for the problem Laplace1(a), our method \eqref{news} with a large $s$ outperforms the DY and SDC methods when the accuracy level is high. Moreover, with proper $(h,s)$, it performs better than the ABB$_{\min 2}$ method, which dominates the DY and SDC methods. For the problem Laplace1(b), our method \eqref{news} is also competitive with the compared methods.
Table \ref{tbnLap2} shows similar results as the former one. However, for the problem Laplace1(b), our method \eqref{news2} clearly outperforms the DY and SDC methods especially when high
accurate solutions are needed.
Moreover, the method \eqref{news2} is very competitive with the ABB$_{\min 2}$ method in terms of total number of iterations.

\subsection{Bound constrained problems}
This subsection compares our methods A1, A1-BB1 and A1-BB2 with the spectral projected gradient (SPG) method \cite{birgin2000nonmonotone,birgin2014spectral}, which is a nonmonotone projected gradient method
using the Barzilai-Borwein stepsize.


For our methods, the parameter values are set as the following:
\begin{equation*}
  \alpha_{\min}=10^{-30},~\alpha_{\max}=10^{30},~h=10,~s=4,~M=8,~\sigma=10^{-4}.
\end{equation*}
Default parameters were  used for SPG \cite{birgin2014spectral}.
The stopping condition for all
methods is
\begin{equation*}
  \|P_{\Omega}(x_k-g_k)-x_k\|_{\infty}\leq10^{-6}.
\end{equation*}

Our test problem set consists of all bound constrained problems from the CUTEst collection
 \cite{gould2015cutest} with dimension more than $50$. There are $3$ problems for which none of these
 comparison algorithms can solve. Hence, we simply delete them and only $47$ problems are left for our test.

We compare all these algorithms by using the performance profiles of Dolan and Mor\'{e} \cite{dolan2002}
on different metric. In these performance profiles, the vertical axis shows the percentage of the problems
the method solves within the factor $\tau$ of the metric used by the most effective method in this comparison.
Figure \ref{iter} shows the performance profiles on the number of iterations.
It can be observed that our methods A1, A1-BB1 and A1-BB2 clearly outperform SPG in terms of iteration numbers.
We can also see from Figure \ref{nfunc} that the performance gap is even larger in terms of
the number of function evaluations. Moreover, Figure \ref{time} shows that our methods are also better than
SPG in terms of the overall CPU time.

\begin{figure}[h]
  \centering
  \includegraphics[width=0.6\textwidth,height=0.47\textwidth]{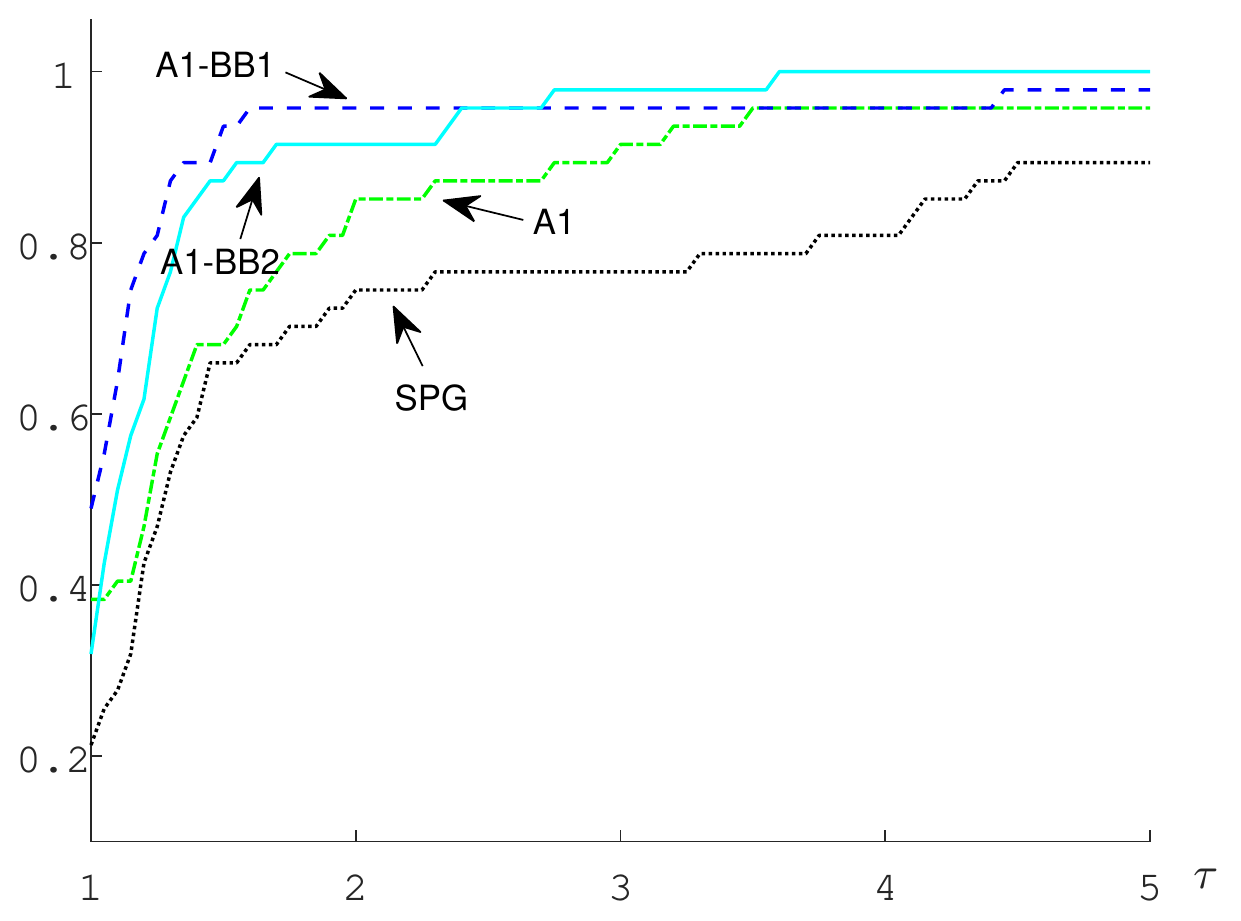}\\
  \caption{Performance profiles, iteration metric, 47 CUTEst bound constrained problems.}\label{iter}
\end{figure}

\begin{figure}[h]
  \centering
  \includegraphics[width=0.6\textwidth,height=0.47\textwidth]{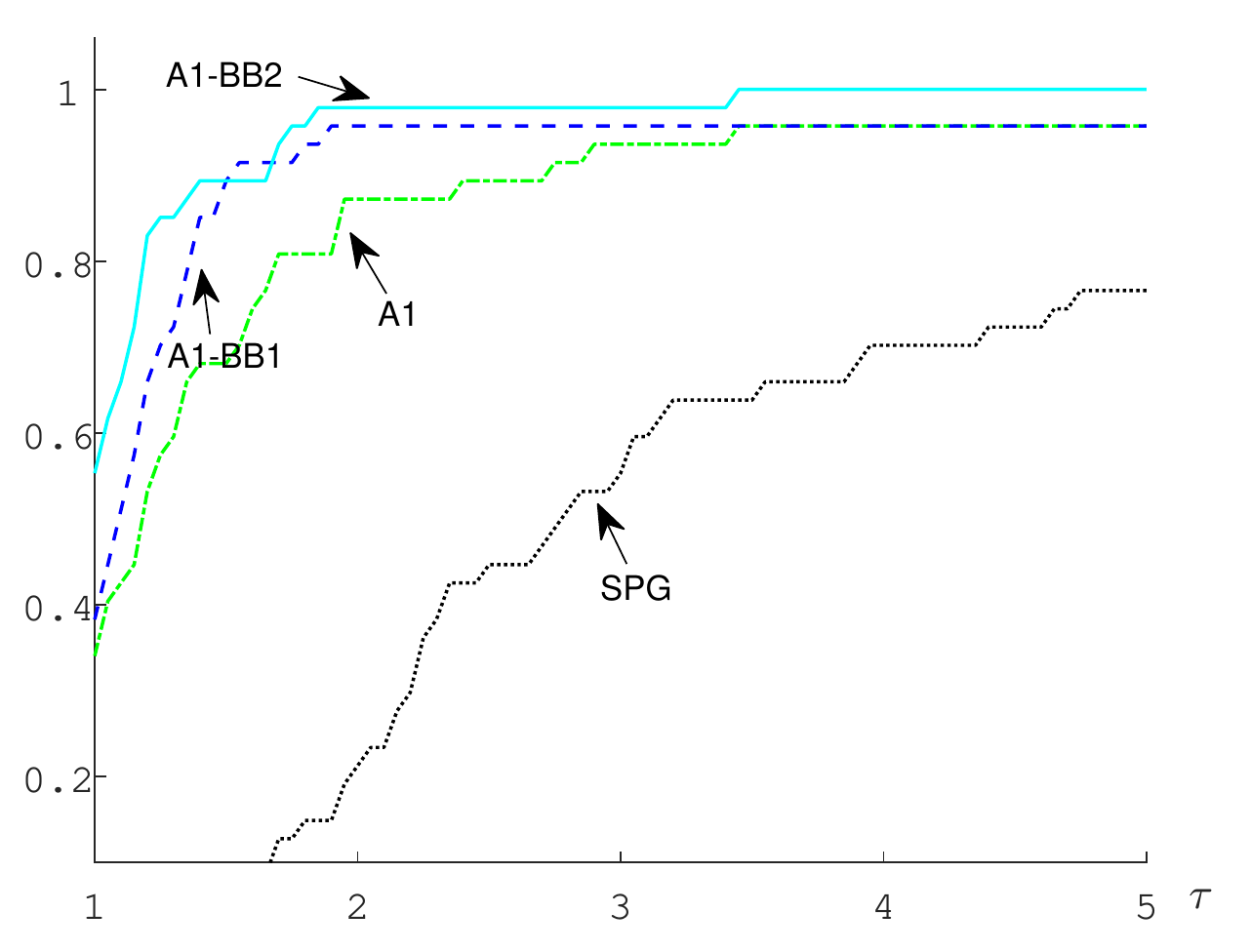}\\
  \caption{Performance profiles, function evaluation metric, 47 CUTEst bound constrained problems.}\label{nfunc}
\end{figure}

\begin{figure}[h]
  \centering
  \includegraphics[width=0.6\textwidth,height=0.47\textwidth]{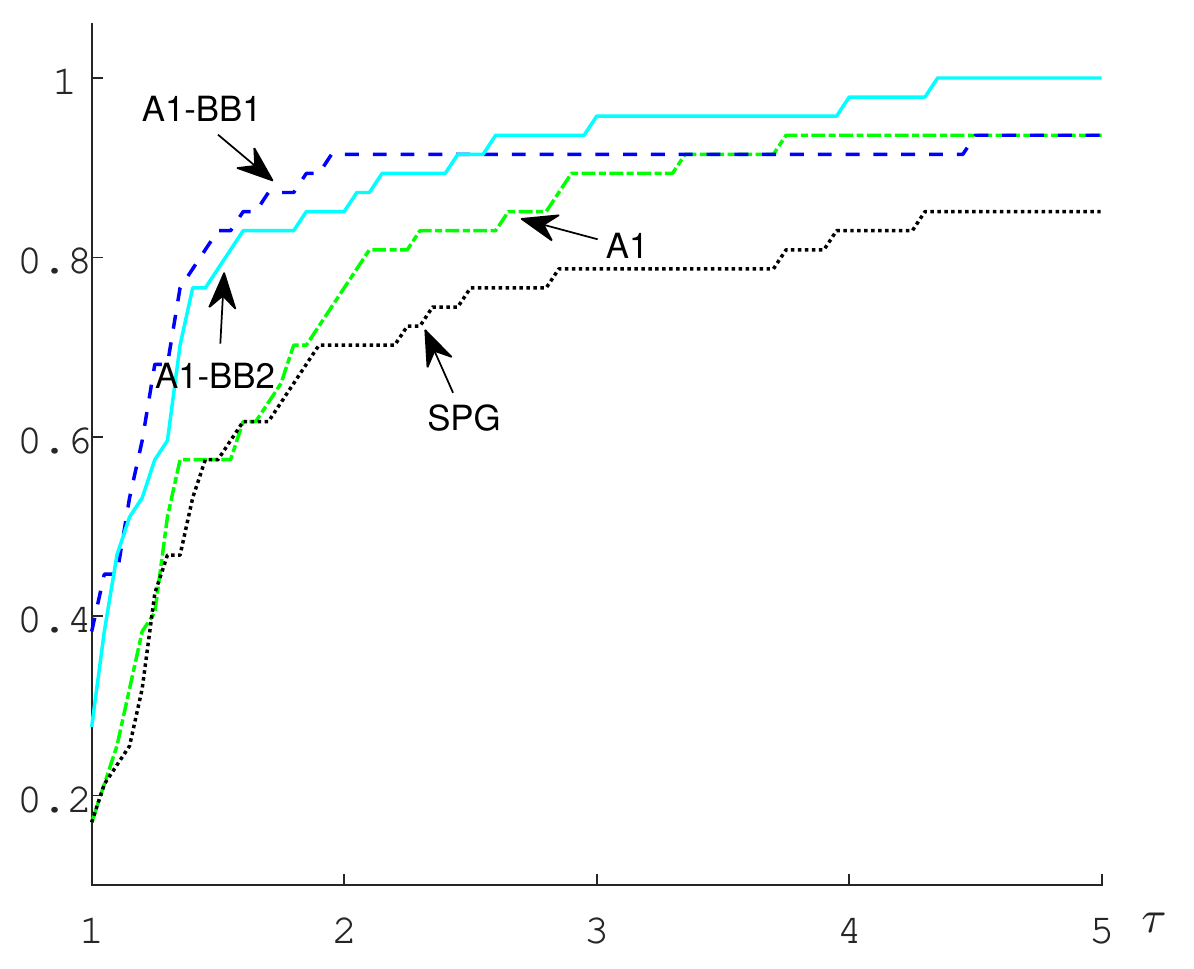}\\
  \caption{Performance profiles, CPU time metric, 47 CUTEst bound constrained problems.}\label{time}
\end{figure}

\section{Conclusions}\label{sec6}
Based on the asymptotic optimal stepsize, we have proposed a new monotone gradient method, which employs a new stepsize that converges to the reciprocal of the largest eigenvalue of the Hessian of the objective function. A nonmonotone variant of this method has been proposed as well. $R$-linear convergence of the proposed methods has been established for minimizing strongly convex quadratic
functions. Our numerical experiments on minimizing quadratic functions show that the proposed methods are very effective with other recent successful gradient methods.

By making use of projected gradient strategy and the Dai-Zhang nonmonotone line search \cite{dai2001adaptive}, the proposed methods are extended for solving general bound constrained optimization.
In addition, we have also proposed two variants of those methods based on the Barzilai-Borwein stepsizes. Numerical comparisons with the spectral projected gradient (SPG) method \cite{birgin2000nonmonotone,birgin2014spectral} on bound constrained problems from the CUTEst collection show that these new methods are very promising for solving bound constrained optimization.



%
%
%
%

\section*{Funding}

This work was supported by the National Natural Science Foundation of
China (11701137, 11631013, 11671116), by the National 973 Program of China
(2015CB856002), by the China Scholarship Council (No. 201806705007), and by the
USA National Science Foundation (1522654, 1819161).

%
%
%
%


\end{document}